\RequirePackage{ifpdf}
\ifpdf 
\documentclass[pdftex]{BCC}
\else
\documentclass{BCC}
\fi

\usepackage{amsmath}
\usepackage{amsthm} 
\usepackage{amsfonts} 
\usepackage{amssymb} 
\usepackage{graphicx}
\usepackage{appendix}
\usepackage{hyperref}
\usepackage{bbm}
\usepackage{mathrsfs}
\usepackage{pgf,tikz}
\usepackage{mathrsfs}
\usepackage{stmaryrd}
\usepackage{enumitem}
\usepackage{multicol}
\usepackage{lineno}
\usepackage{parcolumns}
\usepackage[all]{xy}
\usetikzlibrary{arrows}

\newcommand{\D}{\partial}

\numberwithin{equation}{section}

\def\[#1\]{%
  \begin{align*}#1\end{align*}%
}

\newtheorem*{theorem*}{Theorem}

\DeclareMathAlphabet{\mathpzc}{OT1}{pzc}{m}{it}

\begin{document}
\thispagestyle{empty}
\ArticleName{Bigraded cochain complexes and Poisson cohomology}
\ShortArticleName{Bigraded cochain complexes and Poisson cohomology}
\Author{Andr\'es Pedroza$^{1}$, Eduardo Velasco-Barreras$^{2}$, and Yury Vorobiev$^{3}$.}
\AuthorNameForHeading{Andr\'es Pedroza, Eduardo Velasco-Barreras, and Yury Vorobiev.}

\noindent\emph{Facultad de Ciencias, University of Colima}

\noindent\emph{Bernal D\'iaz del Castillo 340, Colima, M\'exico, 28045.}$^{1}$\vspace{6pt}

\noindent\emph{Department of Mathematics, University of Sonora}

\noindent\emph{Rosales y Blvd. Luis Encinas, Hermosillo, M\'exico, 83000.}$^{2,3}$\vspace{6pt}

\noindent E-mail adresses: \emph{\href{mailto:andres_pedroza@ucol.mx}{andres\_pedroza@ucol.mx}$^{1}$,\href{mailto:lalo.velasco@mat.uson.mx}{lalo.velasco@mat.uson.mx}$^{2}$,\href{mailto:yurimv@guaymas.uson.mx}{yurimv@guaymas.uson.mx}$^{3}$}.

\Abstract{We present an algebraic framework for the computation of low-degree cohomology of a class of bigraded complexes which arise in Poisson geometry around (pre)symplectic leaves. We also show that this framework can be applied to the more general context of Lie algebroids. Finally, we apply our results to compute the low-degree cohomology in some particular cases.}
\Classification{
18G35; 
14F43; 
53C05; 
53C12; 
53D17. 
}

\paragraph{Acknowledgments.} The authors are very grateful to Jes\'us F. Espinoza, Rafael Ramos Figueroa, and Ricardo A. S\'aenz for helpful comments on some aspects of this work. Partially supported by Consejo Nacional de Ciencia y Tecnolog\'ia (CONACYT) under the grant 219631.


\section{Introduction}\label{intro}

In this paper, we study the cohomology of a class of brigaded cochain complexes arising in several geometric contexts related to Poisson geometry and its applications.

Recall that every Poisson manifold $(M,\Pi)$ induces on the algebra $\Gamma(\wedge^{\bullet}TM)$ of multivector fields a coboundary operator $\operatorname{d}_{\Pi}:\Gamma(\wedge^{\bullet}TM)\rightarrow\Gamma(\wedge^{\bullet+1}TM)$, given as the adjoint of $\Pi$ with respect the Schouten-Nijenhuis bracket, $\operatorname{d}_{\Pi}A:=[\Pi,A]$. The resulting cochain complex $(\Gamma(\wedge^{\bullet}TM),\operatorname{d}_{\Pi})$ is the so-called \emph{Lichnerowicz Poisson complex}, and its cohomology $H^{\bullet}_{\Pi}(M)$ is the \emph{Poisson cohomology} of the Poisson manifold $(M,\Pi)$. 
It is also important to point out that the Lichnerowicz-Poisson complex of a Poisson manifold coincides with the cochain complex of its corresponding cotangent Lie algebroid.

Dirac structures are generalizations of Poisson manifolds. A Dirac structure on $M$ is a maximally isotropic subbundle $D\subset\mathbb{T}M$ of the Pontryagin bundle $\mathbb{T}M:=TM\oplus T^{*}M$ which is closed under the Dorfman bracket. The Dorfman bracket induces on $D$ a Lie algebroid structure, so its dual exterior algebra is endowed with a cochain complex structure $(\wedge^{\bullet}D^{*},\operatorname{d}_{D})$. Hence, it makes sense to consider the Lie algebroid cohomology of $D$. In particular, the graph of every Poisson structure $\Pi$ on $M$ is a Dirac structure which is isomorphic to $T^{*}M$ as a Lie algebroid.

Unlike some other cohomological theories, the Lie algebroid cohomology of Poisson or Dirac manifolds is in general hard to compute. Only few general results are known for the computation of the cohomology of certain classes of Poisson manifolds \cite{Conn-85,Lan2-16,Lan1-16,Va-90,VK-88}, as well as some specific examples \cite{AG,Mon-02,Nak-95,Xu-92}.

In this work we present a framework which allows us to describe the Dirac and Poisson cohomology around presymplectic leaves, which is based on the coupling method for Dirac and Poisson manifolds \cite{DuWa-04,Va-04,Va2-04,Vo-01,Vo-05}. Recall that, in a tubular neighborhood $N\overset{\pi}{\rightarrow}S$ of a presymplectic leaf of a Dirac manifold $(M,D)$, the Dirac structure $D|_{N}$ is fully described by a triple of geometric data $(\gamma,\sigma,P)$ consisting of an Ehresmann connection $\gamma$, a horizontal 2-form $\sigma$, and a vertical Poisson structure $P$ on $N$ satisfying some integrability conditions. In particular, the geometric data induce bigraded operators $\D_{0,1}^{P}$, $\D_{1,0}^{\gamma}$, and $\D_{2,-1}^{\sigma}$ on the bigraded algebra $\mathcal{C}^{\bullet,\bullet}:=\Gamma(\wedge^{\bullet}T^{*}S)\otimes_{C^{\infty}(S)}\Gamma(\wedge^{\bullet}\ker\pi_{*})$ of differential forms on $S$ with values on vertical multivector fields on $N$. Then the integrability conditions for $(\gamma,\sigma,P)$ imply that $\D:=\D_{0,1}^{P}+\D_{1,0}^{\gamma}+\D_{2,-1}^{\sigma}$ is a coboundary operator on $\mathcal{C}$ such that the corresponding cochain complex $(\mathcal{C},\D)$ is isomorphic to the complex $(\wedge^{\bullet}D^{*},\operatorname{d}_{D})$ on $N$ \cite{CrMa-13,Marcut-13}.

Algebraically, our previous discussion means that the Lie algebroid cohomology of a Dirac or Poisson manifold around a presymplectic leaf can be described in the framework of cochain complexes $(\mathcal{C},\D)$ with a bigrading $\mathcal{C}^{\bullet,\bullet} = \bigoplus_{p,q\in\mathbb{Z}}\mathcal{C}^{p,q}$ such that the coboundary operator takes the form $\D = \D_{0,1} + \D_{1,0} + \D_{2,-1}$.

The most important contribution of this work is to provide a scheme for the computation of Poisson cohomology in the semilocal context. To this end, we have derived a general procedure which allows to compute the cohomology of a bigraded complex $(\mathcal{C}^{\bullet,\bullet},\D)$ with $\D = \D_{0,1} + \D_{1,0} + \D_{2,-1}$, as described above. Here we present the results we have obtained in the cases of first, second, and third cohomology, but our procedure may be applied to derive similar results in any degree. In particular, we have recovered the results developed in \cite{VeVo-18} for the cohomology of degree 1.

Among our main results, we mention the following.

\begin{theorem*}[First Cohomology]
We have the following short exact sequence
\[
\xymatrix@=1.4em{
0\ar[r]&H^{1}(\mathcal{N}_{0},\overline{\D})\ar@{->}[r]&H^{1}(\mathcal{C},\D)\ar@{->}[r]&\frac{\ker(\rho_{1})}{B^{1}(\mathcal{C}^{0,\bullet},\D_{0,1})}\ar[r]&0,
}
\]
which describe the first cohomology of a bigraded cochain complex $(\mathcal{C}^{\bullet,\bullet},\D)$.
\end{theorem*}

This result on the first cohomology of $(\mathcal{C},\D)$ is the bottom row of the diagram appearing in Theorem \ref{teo:H1}, and involves the map $\rho_{1}:\mathcal{A}^{1}\rightarrow H^{2}(\mathcal{N}_{0},\overline\D)$, which is related to the second cohomology of $(\mathcal{N}_{0},\overline\D)$.

\begin{theorem*}[Second Cohomology]
The following are short exact sequences which allow to describe the second cohomology of a bigraded cochain complex $(\mathcal{C}^{\bullet,\bullet},\D)$:
\begin{gather*}
\xymatrix@=1.4em{
0\ar[r]&\frac{Z^{2}(\mathcal{N}_{0},\overline\D)}{B^{2}(\mathcal{C},\D)\cap\mathcal{C}^{2,0}}\ar@{->}[r]&H^{2}(\mathcal{C},\D)\ar@{->}[r]&\frac{\ker(\rho_{2})}{\mathcal{B}^{2}_{1}}\ar[r]&0,
}\\
\xymatrix@=1.4em{
0\ar[r]&\frac{\ker(\varrho_{2})}{\mathcal{B}^{2}_{1}\cap\mathcal{C}^{1,1}}\ar@{->}[r]&\frac{\ker(\rho_{2})}{\mathcal{B}^{2}_{1}}\ar@{->}[r]&\frac{\mathcal{Z}^{2}_{2}}{B^{2}(\mathcal{C}^{0,\bullet},\D_{0,1})}\ar[r]&0.
}
\end{gather*}
\end{theorem*}

This result consists of the bottom rows of the diagrams appearing in Theorem \ref{teo:H2}. Moreover, it involves the maps $\rho_{2}:\mathcal{A}^{2}\rightarrow H^{3}(\mathcal{N}_{0},\overline\D)$ and $\varrho_{2}:\mathcal{J}^{2}\rightarrow H^{3}(\mathcal{N}_{0},\overline\D)$, related to the third cohomology of $(\mathcal{N}_{0},\overline\D)$. Also, the subspace $\mathcal{Z}^{2}_{2}$ is related to the 3-coboundaries of $(\mathcal{N}_{1},\overline\D)$.

\begin{theorem*}[Third Cohomology]
The third cohomology of $(\mathcal{C}^{\bullet,\bullet},\D)$ is described by the following short exact sequences:
\begin{gather*}
\xymatrix@=1.4em{
0\ar[r]&\frac{Z^{3}(\mathcal{N}_{0},\overline\D)}{B^{3}(\mathcal{C},\D)\cap\mathcal{C}^{3,0}}\ar@{->}[r]&H^{3}(\mathcal{C},\D)\ar@{->}[r]&\frac{\ker(\rho_{3})}{\mathcal{B}^{3}_{1}}\ar[r]&0,
}\\
\xymatrix@=1.4em{
0\ar[r]&\frac{\ker(\varrho_{3})}{\mathcal{B}^{3}_{1}\cap\mathcal{C}^{2,1}}\ar@{->}[r]&\frac{\ker(\rho_{3})}{\mathcal{B}^{3}_{1}}\ar@{->}[r]&\frac{\mathcal{Z}^{3}_{2}}{\mathcal{B}^{3}_{2}}\ar[r]&0,
}\\
\xymatrix@=1.4em{
0\ar[r]&\frac{\mathcal{B}^{3}_{2}\cap\mathcal{C}^{1,2}}{\mathcal{Z}^{3}_{2}\cap\mathcal{C}^{1,2}}\ar@{->}[r]&\frac{\mathcal{Z}^{3}_{2}}{\mathcal{B}^{3}_{2}}\ar@{->}[r]&\frac{\mathcal{Z}^{3}_{3}}{B^{3}(\mathcal{C}^{0,\bullet},\D_{0,1})}\ar[r]&0.
}
\end{gather*}
\end{theorem*}

These short exact sequences are the bottom rows appearing in the diagrams of Theorem \ref{teo:H3}. We note that this result on the third cohomology of $(\mathcal{C}^{\bullet,\bullet},\D)$ involves the maps $\rho_{3}:\mathcal{A}^{3}\rightarrow H^{4}(\mathcal{N}_{0},\overline\D)$ and $\varrho_{3}:\mathcal{J}^{3}\rightarrow H^{4}(\mathcal{N}_{0},\overline\D)$, related to the fourth cohomology of $(\mathcal{N}_{0},\overline\D)$. Also, the subspace $\mathcal{Z}^{3}_{3}$ is related to the 4-coboundaries of $(\mathcal{N}_{1},\overline\D)$.

This class of cochain complexes also appears in the context of transitive Lie algebroids \cite{ItskovYuraKarasev}, regular Poisson manifolds \cite{Va-94}, Poisson foliations, and the de Rham complex of fibred manifolds \cite{Bra-11}. Also, this framework has been also applied in the description of the first cohomology of Poisson manifolds around symplectic leaves \cite{VeVo-18} as well as the modular class of coupling Poisson structures on foliated manifolds \cite{PeVeVo-18}.

\section{The cohomology of a bigraded cochain complex}\label{sec:BCC}

\paragraph{Notations and conventions.} Recall that a cochain complex is a pair $(\mathcal{C}^{\bullet},\D)$ consisting of a graded ($\mathbb{Z}$-graded) $\mathbb{R}$-linear space $\mathcal{C}^{\bullet} = \bigoplus_{k\in\mathbb{Z}}\mathcal{C}^{k}$ and an $\mathbb{R}$-linear operator $\D\in\mathrm{End}^{1}_{\mathbb{R}}(\mathcal{C})$ on $\mathcal{C}$ of degree $1$ such that $\D^{2}=0$.

Suppose that, in addition, $\mathcal{C}$ is a bigraded ($\mathbb{Z}^{2}$-graded) linear space such that the bigrading is compatible with the original $\mathbb{Z}$-grading in the following sense:
\begin{align}\label{eq:bigrading}
  \mathcal{C}^{k} = \bigoplus_{p+q=k}\mathcal{C}^{p,q} && \forall k\in\mathbb{Z}.
\end{align}
We also assume that $\mathcal{C}^{p,q}=\{0\}$ whenever $p$ or $q$ is negative. Moreover, suppose that the coboundary operator $\D$ splits in the sum of three bigraded operators with respect to the bigrading \eqref{eq:bigrading},
\begin{align}\label{eq:DecomPart}
  \D = \D_{2,-1} + \D_{1,0} + \D_{0,1},
\end{align}
where $\D_{i,j}(\mathcal{C}^{p,q})\subseteq\mathcal{C}^{p+i,q+j}$ for $(i,j)\in\{(2,-1),(1,0),(0,1)\}$. The right-hand side of \eqref{eq:DecomPart} is called the \emph{bigraded decomposition} of $\D$.

In terms of the decomposition \eqref{eq:DecomPart}, the coboundary condition $\D^{2}=0$ reads
\begin{align}
  \D_{2,-1}^{2} &= 0,\label{eq:Cob1}\\
  \D_{2,-1}\D_{1,0} + \D_{1,0}\D_{2,-1} &= 0,\label{eq:Cob2}\\
  \D_{2,-1}\D_{0,1} + \D_{0,1}\D_{2,-1} + \D_{1,0}^{2} &=0,\label{eq:Cob3}\\
  \D_{1,0}\D_{0,1} + \D_{0,1}\D_{1,0} &=0,\label{eq:Cob4}\\
  \D_{0,1}^{2} &=0.\label{eq:Cob5}
\end{align}
Here, the left-hand sides of equations \eqref{eq:Cob1}-\eqref{eq:Cob5} are the bigraded components of $\D^{2}$. In particular, \eqref{eq:Cob5} implies that $(\mathcal{C}^{p,\bullet},\D_{0,1})$ is a cochain complex for each $p\in\mathbb{Z}$. For any cochain complex, we use the notation $Z^{\bullet}$, $B^{\bullet}$, and $H^{\bullet}$ to indicate the linear spaces of cocycles, coboundaries, and cohomology, respectively.

\paragraph{Spectral sequence.} Consider the decreasing filtration $F$ of $\mathcal{C}$ given by
\[
F^{p}\mathcal{C} := \bigoplus_{\substack{i,j\in\mathbb{Z}\\i\geq p}}\mathcal{C}^{i,j}
.
\]
For every subspace $S\subseteq\mathcal{C}$, denote $F^{p}S := F^{p}\mathcal{C}\cap S$
. In particular, $F^{p}\mathcal{C}^{k} = \bigoplus_{i\geq p}\mathcal{C}^{i,k-i}$
. Moreover, since $\mathcal{C}^{\bullet,\bullet}$ lies in the first quadrant, we have $F^{0}\mathcal{C}^{k}=\mathcal{C}^{k}$ and $F^{k+1}\mathcal{C}^{k}=\{0\}$, so the filtration is bounded. Furthermore, it follows from the bigraded decomposition \eqref{eq:DecomPart} that $\D(F^{p}\mathcal{C})\subseteq F^{p}\mathcal{C}$ for all $p\in\mathbb{Z}$. Hence, the triple $(\mathcal{C}^{\bullet},\D,F)$ is a graded filtered complex.

Let $(E^{\bullet,\bullet}_{r},d_{r})$ be the spectral sequence associated with $(\mathcal{C}^{\bullet},\D,F)$, that is, for each $p,q,r\in\mathbb{Z}$, $E^{p,q}_{r} := \frac{Z^{p,q}_{r}+F^{p+1}\mathcal{C}^{p+q}}{B^{p,q}_{r-1}+F^{p+1}\mathcal{C}^{p+q}}$ \cite[Eq. (2.46)]{DZ}, where
\[
Z^{p,q}_{r}:=F^{p}\mathcal{C}^{p+q}\cap\D^{-1}(F^{p+r}\mathcal{C}^{p+q}), && B^{p,q}_{r-1}:=F^{p}\mathcal{C}^{p+q}\cap\D(F^{p-r+1}\mathcal{C}^{p+q}),
\]
the sums $Z^{p,q}_{r}+F^{p+1}\mathcal{C}^{p+q}$, and $B^{p,q}_{r-1}+F^{p+1}\mathcal{C}^{p+q}$ are as $\mathbb{R}$-vector subspaces of $\mathcal{C}^{\bullet}$, and $d_{r}:E^{p,q}_{r}\rightarrow E^{p+r,q+1-r}_{r}$ is induced by the restriction of $\D$ to $Z^{p,q}_{r}$. In particular, $E^{p,q}_{0} = \frac{F^{p}\mathcal{C}^{p+q}}{F^{p+1}\mathcal{C}^{p+q}} \cong \mathcal{C}^{p,q}$, so $(E^{\bullet,\bullet}_{r},d_{r})$ is a first quadrant spectral sequence. Therefore, $E^{p,q}_{N}=E^{p,q}_{\infty}$ for all $N\geq\max\{p+1,q+2\}$, where
\[
E^{p,q}_{\infty} := \frac{Z^{p+q}(\mathcal{C},\D)\cap F^{p}\mathcal{C}^{p+q}+F^{p+1}\mathcal{C}^{p+q}}{B^{p+q}(\mathcal{C},\D)\cap F^{p}\mathcal{C}^{p+q}+F^{p+1}\mathcal{C}^{p+q}}.
\]

Since the filtration $F$ is bounded, the spectral sequence converges to the cohomology of $(\mathcal{C}^{\bullet},\D)$. Furthermore, taking into account that $(\mathcal{C}^{\bullet},\D)$ is an $\mathbb{R}$-vector space, we get the following splitting for the $k$-th cohomology of $(\mathcal{C}^{\bullet},\D)$:
\begin{align}\label{eq:SplitHk1}
  H^{k}(\mathcal{C},\D) \cong \bigoplus_{p+q=k} E^{p,q}_{\infty}.
\end{align}

In what follows, we give a more explicit description of the summands in the splitting \eqref{eq:SplitHk1}. For each $q\in\mathbb{Z}$, define $G^{q}\mathcal{C} := \bigoplus_{\substack{i,j\in\mathbb{Z}\\j\geq q}}\mathcal{C}^{i,j}$, and consider the projection
\[
\pi_{q}:\mathcal{C}^{\bullet}\rightarrow G^{q}\mathcal{C}
\]
along the splitting induced by the bigrading. In particular, $\pi_{q}=\operatorname{Id}_{\mathcal{C}}$ if $q\leq0$. For simplicity, we use the same notation for the restriction of $\pi_{q}$ to any subspace of $\mathcal{C}$.

\begin{lemma}\label{lemma:step1}
  For each $p,q\in\mathbb{Z}$ such that $p+q=k$, we have $E^{p,q}_{\infty}\cong\frac{\pi_{q}(Z^{k}(\mathcal{C},\D))\cap\mathcal{C}^{p,q}}{\pi_{q}(B^{k}(\mathcal{C},\D))\cap\mathcal{C}^{p,q}}$.
\end{lemma}

\begin{proof}
Consider the projection $\operatorname{pr}_{p,q}:\mathcal{C}\rightarrow\mathcal{C}^{p,q}$ along the splitting \eqref{eq:bigrading}. Observe that, for each subspace $S\subset\mathcal{C}^{k}$, we have
\[
S\cap F^{p}\mathcal{C}^{k}+F^{p+1}\mathcal{C}^{k} = \operatorname{pr}_{p,q}(S\cap F^{p}\mathcal{C}^{k})\oplus F^{p+1}\mathcal{C}^{k}.
\]
On the other hand, it is straightforward to verify that every element of $\operatorname{pr}_{p,q}(S\cap F^{p}\mathcal{C}^{k})$ is of the form $y_{p,q}$, for some $y\in S$ with bigraded decomposition $y=\sum_{i\geq p}y_{i,k-i}$. Thus,
\[
\operatorname{pr}_{p,q}(S\cap F^{p}\mathcal{C}^{k})=\pi_{q}(S)\cap\mathcal{C}^{p,q}.
\]
Setting $S=Z^{k}(\mathcal{C},\D)$ and $S=B^{k}(\mathcal{C},\D)$, we get
\begin{align*}
E^{p,q}_{\infty} = \frac{Z^{k}(\mathcal{C},\D)\cap F^{p}\mathcal{C}^{k}+F^{p+1}\mathcal{C}^{k}}{B^{k}(\mathcal{C},\D)\cap F^{p}\mathcal{C}^{k}+F^{p+1}\mathcal{C}^{k}} &= \frac{\operatorname{pr}_{p,q}(Z^{k}(\mathcal{C},\D)\cap F^{p}\mathcal{C}^{k})\oplus F^{p+1}\mathcal{C}^{k}}{\operatorname{pr}_{p,q}(B^{k}(\mathcal{C},\D)\cap F^{p}\mathcal{C}^{k})\oplus F^{p+1}\mathcal{C}^{k}}\\
&= \frac{(\pi_{q}(Z^{k}(\mathcal{C},\D))\cap\mathcal{C}^{p,q})\oplus F^{p+1}\mathcal{C}^{k}}{(\pi_{q}(B^{k}(\mathcal{C},\D))\cap\mathcal{C}^{p,q})\oplus F^{p+1}\mathcal{C}^{k}} \cong \frac{\pi_{q}(Z^{k}(\mathcal{C},\D))\cap\mathcal{C}^{p,q}}{\pi_{q}(B^{k}(\mathcal{C},\D))\cap\mathcal{C}^{p,q}}.
\end{align*}
\end{proof}

\paragraph{Splittings for cocycles and coboundaries.} We now derive similar splittings for the spaces of $k$-cocycles and $k$-coboundaries. Observe that for each subspace $S\subset\mathcal{C}^{k}$, we have a family of short exact sequences, given by
\begin{align}\label{eq:ExactS}
0\rightarrow \pi_{q}(S)\cap\mathcal{C}^{p,q} \hookrightarrow \pi_{q}(S) \overset{\pi_{q+1}}{\rightarrow} \pi_{q+1}(S) \rightarrow0, \qquad0
\leq q\leq k-1,
\end{align}
where $p:=k-q$. In the case of cocycles and coboundaries, we get the following result.

\begin{proposition}\label{prop:Diagram}
  For each $p,q\in\mathbb{Z}$, with $p+q=k$, we have the following commutative diagram with exact rows and columns which all describe the spaces of $k$-coboundaries $B^{k}(\mathcal{C},\D)$, $k$-cocycles $Z^{k}(\mathcal{C},\D)$, and $k$-cohomology $H^{k}(\mathcal{C},\D)$:
  \[
    \xymatrix@=1.4em{
    &0\ar[d]&0\ar[d]&0\ar[d]\\
    0\ar[r]&\pi_{q}(B^{k}(\mathcal{C},\D))\cap\mathcal{C}^{p,q}\ar@{^{(}->}[d]\ar@{^{(}->}[r]&\pi_{q}(B^{k}(\mathcal{C},\D))\ar@{^{(}->}[d]\ar@{->>}[r]^{\pi_{q+1}}&\pi_{q+1}(B^{k}(\mathcal{C},\D))\ar@{^{(}->}[d]\ar[r]&0\\
    0\ar[r]&\pi_{q}(Z^{k}(\mathcal{C},\D))\cap\mathcal{C}^{p,q}\ar@{->>}[d]\ar@{^{(}->}[r]&\pi_{q}(Z^{k}(\mathcal{C},\D))\ar@{->>}[d]\ar@{->>}[r]^{\pi_{q+1}}&\pi_{q+1}(Z^{k}(\mathcal{C},\D))\ar@{->>}[d]\ar[r]&0.\\
    0\ar[r]&\frac{\pi_{q}(Z^{k}(\mathcal{C},\D))\cap\mathcal{C}^{p,q}}{\pi_{q}(B^{k}(\mathcal{C},\D))\cap\mathcal{C}^{p,q}}
    \ar[d]\ar@{>->}[r]&\frac{\pi_{q}(Z^{k}(\mathcal{C},\D))}{\pi_{q}(B^{k}(\mathcal{C},\D))}\ar[d]\ar@{->>}[r]&\frac{\pi_{q+1}(Z^{k}(\mathcal{C},\D))}{\pi_{q+1}(B^{k}(\mathcal{C},\D))}\ar[r]\ar[d]&0\\
    &0&0&0}
  \]
\end{proposition}

Here, the mappings from the second to the third row are the canonical projections, and the maps $\frac{\pi_{q}(Z^{k}(\mathcal{C},\D))\cap\mathcal{C}^{p,q}}{\pi_{q}(B^{k}(\mathcal{C},\D))\cap\mathcal{C}^{p,q}} \rightarrow\frac{\pi_{q}(Z^{k}(\mathcal{C},\D))}{\pi_{q}(B^{k}(\mathcal{C},\D))}$ and $\frac{\pi_{q}(Z^{k}(\mathcal{C},\D))}{\pi_{q}(B^{k}(\mathcal{C},\D))}\rightarrow\frac{\pi_{q+1}(Z^{k}(\mathcal{C},\D))}{\pi_{q+1}(B^{k}(\mathcal{C},\D))}$ are defined in such a way that the lower $2\times2$ blocks commute.

\begin{proof}[Proof of Proposition \ref{prop:Diagram}]
  The upper $2\times2$ diagrams are clearly commutative because the arrows $\hookrightarrow$ are natural inclusions, and the arrows with a $\pi_{q+1}$ are the restriction of the same mapping. On the other hand, the exactness of the first row is obtained from \eqref{eq:ExactS} by setting $S:=B^{k}(\mathcal{C},\D)$. Similarly, the exactness of the second row follows from setting $S:=Z^{k}(\mathcal{C},\D)$ in \eqref{eq:ExactS}. Moreover, each column is exact by definition. Finally, the exactness of the last row follows from the commutativity and the exactness of the rest of the diagram.
\end{proof}

\begin{corollary}\label{cor:Splits}
  The coboundary, cocycle, and cohomology spaces of degree $k$ admit the following splittings:
  \begin{align*}
    B^{k}(\mathcal{C},\D)\cong\bigoplus_{p+q=k}\pi_{q}(B^{k}(\mathcal{C},\D))\cap\mathcal{C}^{p,q}, \qquad& Z^{k}(\mathcal{C},\D)\cong\bigoplus_{p+q=k}\pi_{q}(Z^{k}(\mathcal{C},\D))\cap\mathcal{C}^{p,q},
     && \text{and} && \\ H^{k}(\mathcal{C},\D)\cong\bigoplus_{p+q=k}&\tfrac{\pi_{q}(Z^{k}(\mathcal{C},\D))\cap\mathcal{C}^{p,q}}{\pi_{q}(B^{k}(\mathcal{C},\D))\cap\mathcal{C}^{p,q}}.
  \end{align*}
\end{corollary}

Note that the splitting for $H^{k}(\mathcal{C},\D)$ in Corollary \ref{cor:Splits} coincides with \eqref{eq:SplitHk1} under Lemma \ref{lemma:step1}

\begin{remark}
  Every result of this part is valid if the bigraded decomposition of $\D$ has the more general form $\D=\sum_{r\geq0}\D_{r,1-r}$. Moreover, Lemma \ref{lemma:step1} and Proposition \ref{prop:Diagram} still hold if $(\mathcal{C}^{\bullet},\D)$ is a cochain complex over a ring, while the non-canonical splittings in \eqref{eq:SplitHk1} and in Corollary \ref{cor:Splits} only hold in the vector spaces category.
\end{remark}

Otherwise stated, we assume in what follows that $(\mathcal{C}^{\bullet},\D)$ is a cochain complex over a ring $\mathcal{R}$. In the cases when we require that the ring of scalars is a field, this condition will be explicitly indicated.

\section{Describing the cohomology}\label{sec:BigradedDescrip}

In this Section, we introduce some useful objects which allow us to improve our description of the splittings in Corollary \ref{cor:Splits} and the diagram of Proposition \ref{prop:Diagram}.

\paragraph{The null subcomplexes.} For simplicity, for each $p,q,k\in\mathbb{Z}$ we denote
\[
\ker^{k}(\D_{i,j}):=\ker(\D_{i,j}:\mathcal{C}^{k}\rightarrow\mathcal{C}^{k+1}), && \text{and} && \ker^{p,q}(\D_{i,j}):=\ker(\D_{i,j}:\mathcal{C}^{p,q}\rightarrow\mathcal{C}^{p+i,q+j}),
\]
for all $(i,j)\in\{(0,1),(1,0),(2,-1)\}$. Let us denote
\[
\mathcal{N}:=\ker(\D_{0,1}:\mathcal{C}\rightarrow\mathcal{C})\cap\ker(\D_{2,-1}:\mathcal{C}\rightarrow\mathcal{C}),
\]
$\mathcal{N}^{k}:=\ker^{k}(\D_{0,1})\cap\ker^{k}(\D_{2,-1})$, $\mathcal{N}^{p,q}:=\ker^{p,q}(\D_{0,1})\cap\ker^{p,q}(\D_{2,-1})$, and $\mathcal{N}_{q}:=\bigoplus_{p\in\mathbb{Z}}\mathcal{N}^{p-q,q}$. Since $\D_{0,1}$ and $\D_{2,-1}$ are bigraded operators, we have $\mathcal{N}=\bigoplus_{k\in\mathbb{Z}}\mathcal{N}^{k}$ and $\mathcal{N}^{k}=\bigoplus_{p+q=k}\mathcal{N}^{p,q}$. Moreover,

\begin{lemma}\label{lemma:SubCpx}
  The graded $\mathcal{R}$-module $\mathcal{N}$ is a cochain subcomplex of $(\mathcal{C},\D)$. Moreover, for each $q\in\mathbb{Z}$, $\mathcal{N}_{q}$ is also a cochain subcomplex of $(\mathcal{C},\D)$.
\end{lemma}

\begin{proof}
  Since $\mathcal{N}=\bigoplus_{q\in\mathbb{Z}}\mathcal{N}_{q}$, it suffices to show that each $\mathcal{N}_{q}$ is a cochain subcomplex of $(\mathcal{C},\D)$. By definition, $\D_{0,1}$ and $\D_{2,-1}$ vanish on $\mathcal{N}_{q}$. Thus,
  \[
  \D(\mathcal{N}^{p-q,q})=\D_{1,0}(\mathcal{N}^{p-q,q})\subseteq\D_{1,0}(\mathcal{C}^{p-q,q})\subseteq\mathcal{C}^{(p+1)-q,q}.
  \]
  To complete the proof, we just need to verify that $\D_{1,0}(\mathcal{N})\subseteq\mathcal{N}$. Fix $\eta\in\mathcal{N}$. Then, $\D_{2,-1}\eta=0$ and $\D_{0,1}\eta=0$. By applying equations \eqref{eq:Cob2} and \eqref{eq:Cob4},
  \[
    \D_{2,-1}(\D_{1,0}\eta) = -\D_{1,0}\D_{2,-1}\eta = 0,  && \text{and} && \D_{0,1}(\D_{1,0}\eta) = -\D_{1,0}\D_{0,1}\eta = 0,
  \]
  proving that $\D_{1,0}\eta\in\mathcal{N}$. Thus, $\D(\mathcal{N}_{q})\subseteq\mathcal{N}_{q}$, as claimed.
\end{proof}

We denote by $\overline{\D}:=\D|_{\mathcal{N}}$ the coboundary operator on $\mathcal{N}$. We use the same notation for any of the cochain subcomplexes $\mathcal{N}_{q}$. The cochain complexes $(\mathcal{N},\overline{\D})$ and $(\mathcal{N}_{q},\overline{\D})$ are called \emph{the null subcomplexes} of $(\mathcal{C},\D)$. Finally, recall that $\mathcal{C}^{p,q}=\{0\}$ whenever $p$ or $q$ is negative. In particular, $\mathcal{C}^{\bullet,0}\subseteq\ker(\D_{2,-1})$. Therefore, $\mathcal{N}_{0}=\ker(\D_{0,1}:\mathcal{C}^{\bullet,0}\rightarrow\mathcal{C}^{\bullet,1})$. In other words, for $q=0$, the restriction of $\D_{1,0}$ to the $\D_{0,1}$-cocycles of bidegree $(p,0)$ gives the null subcomplex $\mathcal{N}_{0}$.

\paragraph{Pre-coboundaries and pre-cocycles.} Recall that the terms appearing in the upper row of the diagrams of Proposition \ref{prop:Diagram} are of the form $\pi_{q}(B^{k}(\mathcal{C},\D))$ or $\pi_{q}(B^{k}(\mathcal{C},\D))\cap\mathcal{C}^{p,q}$, whose elements are obtained by projecting a $k$-coboundary under $\pi_{q}:\mathcal{C}\rightarrow G^{q}\mathcal{C}$. We call the elements of $\pi_{q}(B^{\bullet}(\mathcal{C},\D))$ \emph{pre-coboundaries}, and the elements of $\pi_{q}(B^{\bullet}(\mathcal{C},\D))\cap\mathcal{C}^{p,q}$, \emph{homogeneous pre-coboundaries}. In a similar fashion, we call the elements of $\pi_{q}(Z^{\bullet}(\mathcal{C},\D))$ \emph{pre-cocycles}, and the elements of $\pi_{q}(Z^{\bullet}(\mathcal{C},\D))\cap\mathcal{C}^{p,q}$, \emph{homogeneous pre-cocycles}.

In this part, we give a detailed description of the $\mathcal{R}$-module of pre-cocycles. In fact, although a pre-cocycle is defined as the projection of a cocycle, we describe a bigger $\mathcal{R}$-module containing the cocycles such that the projection of its elements is again a pre-cocycle. In this sense, we have found some degrees of freedom in the construction of pre-cocycles.

Observe that $\eta\in\mathcal{C}^{1}$ is a 1-cocycle if and only if
\[
\D_{2,-1}\eta_{0,1} + \D_{1,0}\eta_{1,0} = 0, && \D_{1,0}\eta_{0,1} + \D_{0,1}\eta_{1,0} = 0, &&
\D_{0,1}\eta_{0,1} = 0.
\]
The left-hand sides of each equation correspond to the bigraded components of $\D\eta$. Similarly, $\eta\in\mathcal{C}^{2}$ is a 2-cocycle if and only if
\[
\D_{2,-1}\eta_{1,1} + \D_{1,0}\eta_{2,0} = 0, && \D_{2,-1}\eta_{0,2} + \D_{1,0}\eta_{1,1} + \D_{0,1}\eta_{2,0} = 0,\\
\D_{1,0}\eta_{0,2} + \D_{0,1}\eta_{1,1} = 0, &&  \D_{0,1}\eta_{0,2} = 0.
\]
In general, for each $\eta\in\mathcal{C}^{k}$ with bigraded components $\eta_{p,q}\in\mathcal{C}^{p,q}$ ($p+q=k$), the bigraded components of $\D\eta$ are
\[
(\D\eta)_{i,j} = \D_{0,1}\eta_{i,j-1}+\D_{1,0}\eta_{i-1,j}+\D_{2,-1}\eta_{i-2,j+1}, && i+j=k+1.
\]
Let us consider the graded $\mathcal{R}$-modules
\begin{align}\label{eq:BigrM}
\mathcal{M} := \{\eta\in\mathcal{C} \mid \D\eta\in B(\mathcal{N},\overline{\D})\},
\end{align}
and
\[
\mathcal{M}^{k} := \{\eta\in\mathcal{C}^{k} \mid (\D\eta)_{i,j}\in B^{k+1}(\mathcal{N}_{j},\overline{\D}),~ i+j=k+1\}.
\]
Then, $\mathcal{M}^{\bullet}=\bigoplus_{k\in\mathbb{Z}}\mathcal{M}^{k}$. Moreover, it is clear that $Z(\mathcal{C},\D)\subseteq\mathcal{M}$. Now, for each $q\in\mathbb{Z}$, define
\[
\mathcal{Z}_{q} := \{\pi_{q}(\eta) \mid \eta\in\mathcal{M}, \text{ and } \pi_{q}(\D\eta)=0\},  && \text{and} && \mathcal{Z}^{k}_{q} := \{\pi_{q}(\eta) \mid \eta\in\mathcal{M}^{k}, \text{ and } \pi_{q}(\D\eta)=0\}.
\]
In other words, the elements of $\mathcal{Z}_{q}\subseteq G^{q}\mathcal{C}$ are of the form $\pi_{q}(\eta)$, for some $\eta\in\mathcal{C}$ satisfying
\[
\D_{0,1}\eta_{i,j-1}+\D_{1,0}\eta_{i-1,j}+\D_{2,-1}\eta_{i-2,j+1}\in
\left\{\begin{aligned}
  \{0\} && \text{if} && j\geq q,\\
  B(\mathcal{N}_{j},\overline{\D})  && \text{if} && j<q.
\end{aligned}\right.
\]
Note that $\mathcal{Z}^{\bullet}_{q}=\bigoplus_{k\in\mathbb{Z}}\mathcal{Z}^{k}_{q}$. Furthermore, we claim that $\mathcal{Z}_{q}$ is precisely the $\mathcal{R}$-module of pre-cocycles in $G^{q}\mathcal{C}$.

\begin{proposition}\label{prop:Precocyc}
  For each $q\in\mathbb{Z}$, we have $\pi_{q}(Z(\mathcal{C},\D)) = \mathcal{Z}_{q}$. In particular, $\xi\in\mathcal{C}^{p,q}$ is a pre-cocycle if and only if $\D_{0,1}\xi=0$ and there exist $\eta\in F^{p}\mathcal{M}^{p+q}$ such that $\eta_{p,q}=\xi$, and $\D_{1,0}\eta_{p,q}+\D_{0,1}\eta_{p+1,q-1}=0$.
\end{proposition}

\begin{proof}
  The inclusion $\pi_{q}(Z(\mathcal{C},\D)) \subseteq \mathcal{Z}_{q}$ simply follows from the already mentioned fact $Z(\mathcal{C},\D)\subseteq\mathcal{M}$. Conversely, pick $\xi\in \mathcal{Z}_{q}^{k}$, of the form $\xi=\sum_{j\geq q}\xi_{j}$, where $\xi_{j}\in\mathcal{C}^{k-j,j}$. Then, there exists $\eta\in\mathcal{M}^{k}$ such that $\pi_{q}(\eta)=\xi$ and $\pi_{q}(\D\eta)=0$. Let $\eta_{j}\in\mathcal{C}^{k-j,j}$ be the bigraded components of $\eta$. The condition $\eta\in\mathcal{M}^{k}$ implies that for each $j<q$ there exists $\eta'_{j}\in\mathcal{N}^{k-j,j}$ such that
  \[
  \D_{0,1}\eta_{j-1}+\D_{1,0}\eta_{j}+\D_{2,-1}\eta_{j+1}=\D_{1,0}\eta'_{j},
  \]
  Finally, set $\widetilde{\xi}:=\eta - \sum_{j<q}\eta'_{j}$. Since $\D_{0,1}\eta'_{j}=0$ and $\D_{2,-1}\eta'_{j}=0$, it is straightforward to verify that $\D\widetilde{\xi}=0$. Furthermore, $\pi_{q}(\widetilde{\xi})=\xi$, which proves that $\xi\in\pi_{q}(Z^{k}(\mathcal{C},\D))$.
\end{proof}

For each $q,k\in\mathbb{Z}$ denote by $\mathcal{B}^{k}_{q} := \pi_{q}(B^{k}(\mathcal{C},\D))$ the $\mathcal{R}$-module of pre-coboundaries. As a consequence, of Propositions \ref{prop:Diagram} and \ref{prop:Precocyc}, we have:

\begin{theorem}\label{teo:Diagram}
  For each $p,q\in\mathbb{Z}$ with $p+q=k$, we have the following commutative diagrams with exact rows and columns describing the coboundary, cocycle, and cohomology of $(\mathcal{C}^{\bullet,\bullet},\D)$:
  \[
    \xymatrix@=1.4em{
    &0\ar[d]&0\ar[d]&0\ar[d]\\
    0\ar[r]&\mathcal{B}^{k}_{q}\cap\mathcal{C}^{p,q}\ar@{^{(}->}[d]\ar@{^{(}->}[r]&\mathcal{B}^{k}_{q}\ar@{^{(}->}[d]\ar@{->>}[r]^{\pi_{q+1}}&\mathcal{B}^{k}_{q+1}\ar@{^{(}->}[d]\ar[r]&0\phantom{.}\\
    0\ar[r]&\mathcal{Z}^{k}_{q}\cap\mathcal{C}^{p,q}\ar@{->>}[d]\ar@{^{(}->}[r]&\mathcal{Z}^{k}_{q}\ar@{->>}[d]\ar@{->>}[r]^{\pi_{q+1}}&\mathcal{Z}^{k}_{q+1}\ar@{->>}[d]\ar[r]&0.\\
    0\ar[r]&\frac{\mathcal{Z}^{k}_{q}\cap\mathcal{C}^{p,q}}{\mathcal{B}^{k}_{q}\cap\mathcal{C}^{p,q}}
    \ar[d]\ar@{>->}[r]&\frac{\mathcal{Z}^{k}_{q}}{\mathcal{B}^{k}_{q}}\ar[d]\ar@{->>}[r]^{\bar\pi_{q+1}}&\frac{\mathcal{Z}^{k}_{q+1}}{\mathcal{B}^{k}_{q+1}}\ar[d]\ar[r]&0\phantom{.}\\
    &0&0&0}
  \]
\end{theorem}

\begin{corollary}\label{cor:Split}
  In the case when $\mathcal{R}$ is a field, we get the following splittings:
    \[
    B^{k}(\mathcal{C},\D)\cong\bigoplus_{p+q=k}\mathcal{B}^{k}_{q}\cap\mathcal{C}^{p,q}, && Z^{k}(\mathcal{C},\D)\cong\bigoplus_{p+q=k}\mathcal{Z}^{k}_{q}\cap\mathcal{C}^{p,q}, && H^{k}(\mathcal{C},\D)\cong\bigoplus_{p+q=k}\frac{\mathcal{Z}^{k}_{q}\cap\mathcal{C}^{p,q}}{\mathcal{B}^{k}_{q}\cap\mathcal{C}^{p,q}}.
    \]
\end{corollary}

\section{The recursive point of view}\label{sec:BigradedRecursive}

Before going further in the description of the diagrams appearing in Theorem \ref{teo:Diagram} in the low-degree case, let us interpret this result from a recursive point of view. This perspective is particularly useful when we are interested in understanding an specific cohomology class of $(\mathcal{C},\D)$.

Recall from Theorem \ref{teo:Diagram} that the cocycles, coboundaries, and cohomology of $(\mathcal{C},\D)$ is described by a family of diagrams, one for each $q=0,1,\ldots,k$. By denoting $\mathcal{H}^{p,q}_{q}:=\frac{\mathcal{Z}^{k}_{q}\cap\mathcal{C}^{p,q}}{\mathcal{B}^{k}_{q}\cap\mathcal{C}^{p,q}}$, $\mathcal{H}^{k}_{q}:=\frac{\mathcal{Z}^{k}_{q}}{\mathcal{B}^{k}_{q}}$, and $p=k-q$, the bottom row of the $q$-th diagram is
\begin{align}\label{eq:Zrow}
\xymatrix{
0\ar[r]&\mathcal{H}^{p,q}_{q}\ar@{^{(}->}[r]&\mathcal{H}^{k}_{q}\ar@{->>}[r]^{\bar\pi_{q+1}}&\mathcal{H}^{k}_{q+1}\ar[r]&0.
}
\end{align}
Now, pick some $\eta\in Z^{k}(\mathcal{C},\D)$, with bigraded decomposition
\[
\eta = \sum_{i+j=k}\eta_{i,j} = \eta_{0,k} + \eta_{1,k-1} + \cdots + \eta_{k,0}.
\]
Since $H^{k}(\mathcal{C},\D) = \mathcal{H}^{k}_{0}$, we have $[\eta]\in\mathcal{H}^{k}_{0}$. Denote $[\eta]_{0}:=[\eta]$, and for each $q=1,\ldots,k$, recursively define $[\eta]_{q}\in\mathcal{H}^{k}_{q}$ by $[\eta]_{q}:=\bar\pi_{q}[\eta]_{q-1}$. It is clear that $[\eta]_{q}$ is well defined for each $q$. Explicitly, we have $[\eta]_{k+1}=0$, $[\eta]_{k}=\eta_{0,k}+\mathcal{B}^{k}_{k}$, and in general
\[
[\eta]_{q} = \sum_{\substack{i+j=k\\j\geq q}}\eta_{k-j,j} + \mathcal{B}^{k}_{q}.
\]

In what follows, let us describe the obstructions for the vanishing of the cohomology class $[\eta]\in H^{k}(\mathcal{C},\D)$. By the relation $[\eta]_{q+1}=\bar\pi_{q+1}[\eta]_{q}$, a necessary condition for $[\eta]_{q}=0$ is that $[\eta]_{q+1}=0$. Conversely, if $[\eta]_{q+1}=0$, then the exactness of \eqref{eq:Zrow} implies that $[\eta]_{q} = \eta_{p,q} + \mathcal{B}^{k}_{q}\cap\mathcal{C}^{p,q}$, where $\eta_{p,q}\in\mathcal{Z}^{k}_{q}\cap\mathcal{C}^{p,q}$. So, under the vanishing of $[\eta]_{q+1}$, the class $[\eta_{p,q}]\in\mathcal{H}^{p,q}_{q}$ is well defined, and the property $[\eta]_{q}=0$ is equivalent to the vanishing of $[\eta_{p,q}]$.

To get more insight in the previous facts, let us describe them in an explicit fashion. Consider the bigraded decomposition $\eta=\sum_{i+j=k}\eta_{i,j}$. Clearly, $\pi_{k}\eta=\eta_{0,k}$, so $[\eta]_{k}=[\eta_{0,k}]\in\mathcal{H}^{0,k}_{k}$. Now, suppose that $[\eta_{0,k}]=0$. Then, there exists $\eta'=\partial\xi\in B^{k}(\mathcal{C},\partial)$ such that $\eta_{0,k} = \eta'_{0,k}$. Since $[\eta]=[\eta-\eta']$, the representative $\eta-\eta'$ is such that the component of bidegree $(0,k)$ vanishes. So, without loss of generality, we may assume that $\eta_{0,k}=0$. Then, $[\eta]_{k-1}=[\eta_{1,k-1}]\in\mathcal{H}^{1,k-1}_{k-1}$. Assuming that $[\eta_{1,k-1}]=0$, there exists $\eta''=\partial\xi'$ such that $\pi_{k-1}\eta''=\eta_{1,k-1}$ (so, in particular, $\eta''_{0,k}=0$). The difference $\eta-\eta''$ is a representative of $[\eta]$ such that the components of bidegree $(0,k)$ and $(1,k-1)$ vanish, so we may assume that $\pi_{k-1}\eta=0$. Thus, $[\eta]_{k-2}=[\eta_{2,k-2}]\in\mathcal{H}^{1,k-1}_{k-1}$, and so on.

In summary, the short exact sequences given by the bottom diagrams can be described in the following way. Given $[\eta]\in H^{k}(\mathcal{C},\D)$, an obstruction to $[\eta]=0$ is the cohomology class $[\eta_{0,k}]\in\mathcal{H}^{0,k}_{k}$. If $[\eta_{0,k}]=0$, then the class $[\eta_{1,k-1}]\in\mathcal{H}^{1,k-1}_{k-1}$ is well defined, and is a new obstruction to the vanishing of $[\eta]$. If in addition $[\eta_{1,k-1}]=0$, then $[\eta_{2,k-2}]\in\mathcal{H}^{2,k-2}_{k-2}$ is well defined and is a new obstruction to the vanishing of $[\eta]$. On every stage, under the vanishing of $[\eta_{p,q}]\in\mathcal{H}^{p,q}_{q}$, the class $[\eta_{p+1,q-1}]\in\mathcal{H}^{p+1,q-1}_{q-1}$ is well defined, independent of the choice of the representative $\eta$, and is an obstruction to $[\eta]=0$. In the last stage, our cohomology class is of the form $[\eta]=[\eta_{k,0}]\in\mathcal{H}^{k,0}_{0}$.

\paragraph{Low degree.} Given $\eta\in Z^{k}(\mathcal{C},\D)$, it is important to remark that $[\eta_{p,q}]\in\mathcal{H}^{p,q}_{q}$ is well defined only in the case when $[\eta_{p-1,q+1}]\in\mathcal{H}^{p-1,q+1}_{q+1}$ also is well defined and vanishes, $[\eta_{p-1,q+1}]=0$. As explained in the previous paragraphs, the vanishing of a cohomology class of degree $k$ is controlled by a sequence of $(k+1)$ ``simpler'' cohomology classes, and each of them is obtained by projecting into the bigraded components of the original one, as long as the previous cohomology class vanishes. Let us illustrate this in some low-degree cases.

If $k=1$, then the cohomology is described by only one short exact sequence,
\[
\xymatrix{
0\ar[r]&\mathcal{H}^{1,0}_{0}\ar@{^{(}->}[r]&H^{1}(\mathcal{C},\D)\ar@{->>}[r]^{\bar\pi_{1}}&\mathcal{H}^{0,1}_{1}\ar[r]&0.
}
\]
In this case, the vanishing of a cohomology class $[\eta]\in H^{1}(\mathcal{C},\D)$ is controlled by at most two cohomology classes. In fact, a necessary condition for $[\eta]=0$ is that $[\eta_{0,1}]\in\mathcal{H}^{0,1}_{1}$ vanishes. Conversely, under $[\eta_{0,1}]=0$, the cohomology class $[\eta_{1,0}]\in\mathcal{H}^{1,0}_{0}$ is well-defined and satisfies $[\eta]=[\eta_{1,0}]$, due to the exactness of the sequence. Thus, if $[\eta_{0,1}]=0$, then the vanishing of $[\eta]$ is equivalent to $[\eta_{1,0}]=0$.

For $k=2$, the cohomology is described by means of two short exact sequences, namely
\[
\xymatrix{
0\ar[r]&\mathcal{H}^{2,0}_{0}\ar@{^{(}->}[r]&H^{2}(\mathcal{C},\D)\ar@{->>}[r]^{\bar\pi_{1}}&\mathcal{H}^{2}_{1}\ar[r]&0,
} &&
\xymatrix{
0\ar[r]&\mathcal{H}^{1,1}_{1}\ar@{^{(}->}[r]&\mathcal{H}^{2}_{1}\ar@{->>}[r]^{\bar\pi_{2}}&\mathcal{H}^{0,2}_{2}\ar[r]&0.
}
\]
In this case, the vanishing of $[\eta]\in H^{2}(\mathcal{C},\D)$ is controlled by at most three cohomology classes. A necessary condition is $[\eta_{0,2}]=0$. Under this condition, the cohomology class $[\eta_{1,1}]\in\mathcal{H}^{1,1}_{1}$ is well defined and satisfies $[\eta_{1,1}]=[\eta]_1$, due to the exactness of the second sequence. In this case, a necessary condition for the vanishing of $[\eta]$ is $[\eta_{1,1}]=0$. Under this condition, the cohomology class $[\eta_{2,0}]\in\mathcal{H}^{2,0}_{0}$ is well defined and satisfies $[\eta_{2,0}]=[\eta]$, due to the exactness of the first sequence. Hence, if $[\eta_{0,2}]=0$ and $[\eta_{1,1}]=0$, then the vanishing of $[\eta]$ is equivalent to $[\eta_{2,0}]=0$.

\section{Cohomology in low degree}

In this section, we describe the diagrams of Theorem \ref{teo:Diagram} for the cases $k=1,2,3$ in more detail.

Following the notation of Section \ref{sec:BigradedDescrip}, observe that the homogeneous pre-cocycles of bidegree $(k,0)$ are just the $k$-cocycles in $(\mathcal{N}_{0},\overline{\D})$,
\begin{align}\label{eq:ZC}
  \mathcal{Z}^{k}_{0}\cap\mathcal{C}^{k,0} = \{\eta\in\mathcal{C}^{k,0}\mid \D_{1,0}\eta=0,\D_{0,1}\eta=0\} = Z^{k}(\mathcal{N}_{0},\overline{\D}).
\end{align}

On the other hand, the homogeneous pre-coboundaries of bidegree $(0,k)$ are precisely the $k$-coboundaries of the complex $(\mathcal{C}^{0,\bullet},\D_{0,1})$,
\begin{align}\label{eq:Cobound}
  \mathcal{B}^{k}_{k}\cap\mathcal{C}^{0,k} = B^{k}(\mathcal{C}^{0,\bullet},\D_{0,1}).
\end{align}

We now refine our description of the $\mathcal{R}$-module $\mathcal{Z}^{k}_{1}$ of pre-cocycles for $q=1$.

\paragraph{The mappings} $\rho:\mathcal{A}\rightarrow H(\mathcal{N},\overline\D)$ \textbf{and} $\varrho:\mathcal{J}\rightarrow H(\mathcal{N},\overline\D)$\textbf{.}
For each $k\in\mathbb{Z}$, consider the linear modules $\mathcal{A}^{k}$ and $\mathcal{J}^{k}$, where
\begin{align}\label{eq:SpacesAB}
\mathcal{A}^{k} := \{\pi_{1}(\eta)\mid \eta\in\mathcal{C}^{k},~\pi_{1}(\D\eta)=0\},\text{ and }
\mathcal{J}^{k} := \mathcal{A}^{k}\cap\mathcal{C}^{k-1,1}.
\end{align}
Explicitly, $\xi\in G^{1}\mathcal{C}^{k}$ lies in $\mathcal{A}^{k}$ if and only if $\pi_{2}(\D\xi)=0$ and there exists $\eta\in\mathcal{C}^{k,0}$ such that $\D_{0,1}\eta+\D_{1,0}\xi_{k-1,1}+\D_{2,-1}\xi_{k-2,2}=0$. In particular, $\xi\in\mathcal{C}^{k-1,1}$ lies in $\mathcal{J}^{k}$ if and only if $\D_{0,1}\xi=0$ and there exists $\eta\in\mathcal{C}^{k,0}$ such that $\D_{0,1}\eta+\D_{1,0}\xi=0$.

\begin{lemma}\label{lemma:SpaceB}
  For each $\eta\in\mathcal{C}^{k}$ such that $\pi_{1}(\D\eta)=0$, one has $\operatorname{pr}_{k+1,0}(\D\eta)\in Z^{k+1}(\mathcal{N}_{0},\overline\D)$.
\end{lemma}

\begin{proof}
  One must show that $\operatorname{pr}_{k+1,0}(\D\eta)\in\ker\D_{1,0}\cap\ker\D_{0,1}$. First note that
  \[
  \operatorname{pr}_{k+1,0}(\D\eta) = \D_{2,-1}\eta_{k-1,1} + \D_{1,0}\eta_{k,0}.
  \]
  By applying \eqref{eq:Cob3} and \eqref{eq:Cob4},
  \begin{align}
  \D_{0,1}(\operatorname{pr}_{k+1,0}(\D\eta)) &= \D_{0,1}\D_{2,-1}\eta_{k-1,1} + \D_{0,1}\D_{1,0}\eta_{k,0}\nonumber\\
   &= -\D_{1,0}^{2}\eta_{k-1,1} - \D_{2,-1}\D_{0,1}\eta_{k-1,1} -\D_{1,0}\D_{0,1}\eta_{k,0}.\label{eq:ProofLemmaSpaceB1}
  \end{align}
  The condition $\pi_{1}(\D\eta)=0$ implies that $\operatorname{pr}_{k,1}(\D\eta)=\D_{0,1}\eta_{k,0} + \D_{1,0}\eta_{k-1,1} + \D_{2,-1}\eta_{k-2,2} = 0$ which, together with \eqref{eq:ProofLemmaSpaceB1}, leads to
  \begin{align}\label{eq:ProofLemmaSpaceB2}
  \D_{0,1}(\operatorname{pr}_{k+1,0}(\D\eta)) &= \D_{0,1}(\operatorname{pr}_{k+1,0}(\D\eta)) + \D_{1,0}(\operatorname{pr}_{k,1}(\D\eta))\nonumber\\
  &= - \D_{2,-1}\D_{0,1}\eta_{k-1,1} + \D_{1,0}\D_{2,-1}\eta_{k-2,2}.
  \end{align}
  Now, from \eqref{eq:Cob2}, we get
  \begin{align}\label{eq:ProofLemmaSpaceB3}
  \D_{0,1}(\operatorname{pr}_{k+1,0}(\D\eta)) = - \D_{2,-1}\D_{0,1}\eta_{k-1,1} - \D_{2,-1}\D_{1,0}\eta_{k-2,2}.
  \end{align}
  Again, from $\pi_{1}(\D\eta)=0$ we get $\operatorname{pr}_{k-1,2}(\D\eta)=\D_{0,1}\eta_{k-1,1} + \D_{1,0}\eta_{k-2,2} + \D_{2,-1}\eta_{k-3,3} = 0$. By \eqref{eq:ProofLemmaSpaceB3},
  \[
  \D_{0,1}(\operatorname{pr}_{k+1,0}(\D\eta)) = \D_{0,1}(\operatorname{pr}_{k+1,0}(\D\eta)) + \D_{2,-1}(\operatorname{pr}_{k-1,2}(\D\eta)) = \D_{2,-1}^{2}\eta_{k-3,3}.
  \]
  Therefore, $\D_{0,1}(\operatorname{pr}_{k+1,0}(\D\eta))=0$, due to \eqref{eq:Cob1}. In a similar fashion, by applying \eqref{eq:Cob2} and \eqref{eq:Cob3} we obtain
  \[
  \D_{1,0}(\operatorname{pr}_{k+1,0}(\D\eta)) = -\D_{2,-1}\D_{1,0}\eta_{k-1,1} - \D_{0,1}\D_{2,-1}\eta_{k,0} - \D_{2,-1}\D_{0,1}\eta_{k,0}.
  \]
  Note that $\D_{2,-1}\eta_{k,0}=0$, due to its negative bidegree. Taking into account that $\operatorname{pr}_{k,1}(\D\eta)=\D_{0,1}\eta_{k,0} + \D_{1,0}\eta_{k-1,1} + \D_{2,-1}\eta_{k-2,2} = 0$, we get
  \[
  \D_{1,0}(\operatorname{pr}_{k+1,0}(\D\eta)) = \D_{2,-1}^{2}\eta_{k-2,2},
  \]
  which is zero because of \eqref{eq:Cob1}.
\end{proof}

Now, by definition, for each $\xi\in\mathcal{A}^{k}$ there exist $\eta\in\mathcal{C}^{k}$ such that $\pi_{1}\eta=\xi$ and $\pi_{1}(\D\eta)=0$. By Lemma \ref{lemma:SpaceB}, $\eta$ induces a cohomology class
\[
[\D_{2,-1}\eta_{k-1,1} + \D_{1,0}\eta_{k,0}]\in H^{k+1}(\mathcal{N}_{0},\overline\D).
\]
We claim that the cohomology class only depends on $\xi$, that is, it is independent of the choice of $\eta$. Indeed, pick another $\widetilde{\eta}\in\mathcal{C}^{k}$ such that $\pi_{1}\widetilde{\eta}=\xi$ and $\pi_{1}(\D\widetilde{\eta})=0$. Since $\eta_{k-1,1}=\widetilde{\eta}_{k-1,1}=\xi_{k-1,1}$, we get
\[
(\D_{2,-1}\widetilde{\eta}_{k-1,1} + \D_{1,0}\widetilde{\eta}_{k,0})-(\D_{2,-1}\eta_{k-1,1} + \D_{1,0}\eta_{k,0}) = \D_{1,0}(\widetilde{\eta}_{k,0}-\eta_{k,0}).
\]
To see that $\eta$ and $\widetilde{\eta}$ induce the same cohomology class, we just need to check that $\widetilde{\eta}_{k,0}-\eta_{k,0}\in\mathcal{N}_{0}$. From $\pi_{1}(\D\eta)=0$ and $\pi_{1}(\D\widetilde{\eta})=0$ we get that $\D_{0,1}\eta + \D_{0,1}\xi_{k,0} =0$ and $\D_{0,1}\widetilde{\eta} + \D_{0,1}\xi_{k,0}=0$. Therefore, $\widetilde{\eta}_{k,0}-\eta_{k,0}\in\ker^{k,0}\D_{0,1}=\mathcal{N}^{k,0}$. Hence, the cohomology class is well defined.

This can be summarized in the following fact.

\begin{lemma}\label{lemma:Rho}
  There exists a well-defined linear map $\rho_{k}:\mathcal{A}^{k}\rightarrow H^{k+1}(\mathcal{N}_{0},\overline\D)$ given by
  \[
  \rho_{k}(\xi) := [\D_{2,-1}\xi_{k-1,1} + \D_{1,0}\eta_{k,0}],
  \]
  where $\eta\in\mathcal{C}^{k}$ is such that $\pi_{1}\eta=\xi$ and $\pi_{1}(\D\eta)=0$. Moreover, we have the identity $\mathcal{Z}^{k}_{1}=\ker(\rho_{k})$.
\end{lemma}

\begin{proof}
  The fact that $\rho_{k}(\xi)$ is well defined follows from our previous discussion, in which we have explained that $[\D_{2,-1}\xi_{k-1,1} + \D_{1,0}\eta_{k,0}]$ only depends of $\xi$. Moreover, the linearity of $\rho_{k}$ follows from the linearity of $\D_{2,-1}$, $\D_{1,0}$, and $\pi_{1}$. So, it is left to show that $\ker(\rho_{k})=\mathcal{Z}^{k}_{1}$. Recall that, by definition, the elements of $\mathcal{Z}^{k}_{1}$ are of the form $\pi_{1}(\eta)$, where $\D_{2,-1}\eta_{k-1,1} + \D_{1,0}\eta_{k,0}\in B^{k+1}(\mathcal{N}_{0},\overline\D)$, and $\pi_{1}(\D\eta)=0$. Therefore,
  \[
  \ker(\rho_{k}) = \{\pi_{1}(\eta)\mid \eta\in\mathcal{C}^{k},~\pi_{1}(\D\eta)=0, \text{ and } \D_{2,-1}\eta_{k-1,1} + \D_{1,0}\eta_{k,0}\in B^{k+1}(\mathcal{N}_{0},\overline\D)\} = \mathcal{Z}^{k}_{1}.
  \]
\end{proof}

For each $k\in\mathbb{Z}_{\geq0}$, define $\varrho_{k}:\mathcal{J}^{k}\rightarrow H^{k+1}(\mathcal{N}_{0},\overline\D)$ by the restriction $\varrho_{k}:=\rho_{k}|_{\mathcal{J}^{k}}$. As a consequence of Lemma \ref{lemma:Rho}, we have
\[
\ker(\varrho_{k})=\mathcal{Z}^{k}_{1}\cap\mathcal{C}^{k-1,1}.
\]

\paragraph{Refining the splittings for the low-degree cohomology.} By applying Theorem \ref{teo:Diagram}, we describe the first, second, and third cohomology of the bigraded cochain complex $(\mathcal{C},\D)$ in terms of the cochain complexes $(\mathcal{C}^{p,\bullet},\D_{0,1})$, $(\mathcal{N}_{q},\overline\D)$, and the mappings $\rho:\mathcal{A}\rightarrow H(\mathcal{N},\overline\D)$ and $\varrho:\mathcal{J}\rightarrow H(\mathcal{N},\overline\D)$ given in Lemmas \ref{lemma:SubCpx} and \ref{lemma:Rho}.

\paragraph{First cohomology.} Here we state our main result on the first cohomology of $(\mathcal{C},\D)$.

\begin{theorem}\label{teo:H1}
We have the following commutative diagram with exact rows and columns,
\[
\xymatrix@=1.4em{
&0\ar[d]&0\ar[d]&0\ar[d]\\
0\ar[r]&B^{1}(\mathcal{N}_{0},\overline{\D})\ar@{^{(}->}[d]\ar@{^{(}->}[r]&B^{1}(\mathcal{C},\D)\ar@{^{(}->}[d]\ar[r]^{\pi_{1}}&B^{1}(\mathcal{C}^{0,\bullet},\D_{0,1})\ar@{^{(}->}[d]\ar[r]&0\\
0\ar[r]&Z^{1}(\mathcal{N}_{0},\overline{\D})\ar[d]\ar@{^{(}->}[r]&Z^{1}(\mathcal{C},\D)\ar[d]\ar[r]^{\pi_{1}}&\ker(\rho_{1})\ar[d]\ar[r]&0.\\
0\ar[r]&H^{1}(\mathcal{N}_{0},\overline{\D})\ar[d]\ar@{->}[r]&H^{1}(\mathcal{C},\D)\ar[d]\ar@{->}[r]&\frac{\ker(\rho_{1})}{B^{1}(\mathcal{C}^{0,\bullet},\D_{0,1})}\ar[d]\ar[r]&0\\
&0&0&0}
\]
\end{theorem}

Observe that this result on the first cohomology of $(\mathcal{C},\D)$ involves the map $\rho_{1}:\mathcal{A}^{1}\rightarrow H^{2}(\mathcal{N}_{0},\overline\D)$, which is related to the second cohomology of $(\mathcal{N}_{0},\overline\D)$.

\begin{proof}[Proof of Theorem \ref{teo:H1}]
The fact that the diagram of Theorem \ref{teo:H1} coincides with the one given in Theorem \ref{teo:Diagram} for $k=1$ follows from equations \eqref{eq:ZC} and \eqref{eq:Cobound}, the definition of $\varrho_{k}$, and from Lemma \ref{lemma:Rho}. We also need the following identity,
\[
\mathcal{B}^{1}_{0}\cap\mathcal{C}^{1,0} = \{\D_{1,0}f\mid \D_{0,1}f=0, f\in\mathcal{C}^{0}\} = B^{1}(\mathcal{N}_{0},\overline\D).
\]
\end{proof}

\begin{corollary}\label{cor:H1}
  In the case when $\mathcal{R}$ is a field, the coboundary, cocycle, and cohomology spaces of degree $1$ admit the following splittings as vector spaces:
  \begin{gather*}
  B^{1}(\mathcal{C},\D)\cong B^{1}(\mathcal{N}_{0},\overline\D)\oplus B^{1}(\mathcal{C}^{0,\bullet},\D_{0,1}), \qquad
  Z^{1}(\mathcal{C},\D)\cong Z^{1}(\mathcal{N}_{0},\overline\D)\oplus \ker(\rho_{1}),\\
  H^{1}(\mathcal{C},\D)\cong H^{1}(\mathcal{N}_{0},\overline\D)\oplus\frac{\ker(\rho_{1})}{B^{1}(\mathcal{C}^{0,\bullet},\D_{0,1})}.
  \end{gather*}
\end{corollary}

Explicitly,
\begin{align*}
B^{1}(\mathcal{N}_{0},\overline\D) &= \{\D_{1,0}f \mid f\in\mathcal{C}^{0,0}, \D_{0,1}f=0\},\\
B^{1}(\mathcal{C}^{0,\bullet},\D_{0,1}) &= \{\D_{0,1}f \mid f\in\mathcal{C}^{0,0}\},\\
\mathcal{A}^{1} &= \{Y\in\mathcal{C}^{0,1}\mid\D_{0,1}Y=0,~\exists\alpha_{Y}\in\mathcal{C}^{1,0}:\D_{0,1}\alpha_{Y}+\D_{1,0}Y=0\},\\
Z^{1}(\mathcal{N}_{0},\overline\D) &= \{\alpha\in\mathcal{C}^{1,0} \mid \D_{0,1}\alpha=0,\D_{1,0}\alpha=0\},\\
\ker(\rho_{1}) &= \{Y\in\mathcal{A}^{1} \mid \D_{2,-1}Y + \D_{1,0}\alpha_{Y}\in B^{1}(\mathcal{N}_{0},\overline\D)\}.
\end{align*}

\paragraph{Second cohomology.} Similarly, the $\mathcal{R}$-modules of cocycles, coboundaries, and cohomology of degree 2 of the bigraded cochain complex $(\mathcal{C},\D)$ are described by the following more explicit diagrams.

\begin{theorem}\label{teo:H2}
We have the following commutative diagrams with exact rows and columns,
\[
\xymatrix@=1.12em{
&0\ar[d]&0\ar[d]&0\ar[d]\\
0\ar[r]&B^{2}(\mathcal{C},\D)\cap\mathcal{C}^{2,0}\ar@{^{(}->}[d]\ar@{^{(}->}[r]&B^{2}(\mathcal{C},\D)\ar@{^{(}->}[d]\ar[r]^{\pi_{1}}&\mathcal{B}^{2}_{1}\ar@{^{(}->}[d]\ar[r]&0\phantom{.}\\
0\ar[r]&Z^{2}(\mathcal{N}_{0},\overline\D)\ar[d]\ar@{^{(}->}[r]&Z^{2}(\mathcal{C},\D)\ar[d]\ar[r]^{\pi_{1}}&\ker(\rho_{2})\ar[d]\ar[r]&0,\\
0\ar[r]&\frac{Z^{2}(\mathcal{N}_{0},\overline\D)}{B^{2}(\mathcal{C},\D)\cap\mathcal{C}^{2,0}}\ar[d]\ar@{->}[r]&H^{2}(\mathcal{C},\D)\ar[d]\ar@{->}[r]&\frac{\ker(\rho_{2})}{\mathcal{B}^{2}_{1}}\ar[d]\ar[r]&0\phantom{.}\\
&0&0&0} ~
\xymatrix@=1.12em{
&0\ar[d]&0\ar[d]&0\ar[d]\\
0\ar[r]&\mathcal{B}^{2}_{1}\cap\mathcal{C}^{1,1}\ar@{^{(}->}[d]\ar@{^{(}->}[r]&\mathcal{B}^{2}_{1}\ar@{^{(}->}[d]\ar[r]^{\pi_{2}}&B^{2}(\mathcal{C}^{0,\bullet},\D_{0,1})\ar@{^{(}->}[d]\ar[r]&0\phantom{.}\\
0\ar[r]&\ker(\varrho_{2})\ar[d]\ar@{^{(}->}[r]&\ker(\rho_{2})\ar[d]\ar[r]^{\pi_{2}}&\mathcal{Z}^{2}_{2}\ar[d]\ar[r]&0.\\
0\ar[r]&\frac{\ker(\varrho_{2})}{\mathcal{B}^{2}_{1}\cap\mathcal{C}^{1,1}}\ar[d]\ar@{->}[r]&\frac{\ker(\rho_{2})}{\mathcal{B}^{2}_{1}}\ar[d]\ar@{->}[r]&\frac{\mathcal{Z}^{2}_{2}}{B^{2}(\mathcal{C}^{0,\bullet},\D_{0,1})}\ar[d]\ar[r]&0\phantom{.}\\
&0&0&0}
\]
\end{theorem}

Observe that this result on the second cohomology of $(\mathcal{C},\D)$ involves the maps $\rho_{2}:\mathcal{A}^{2}\rightarrow H^{3}(\mathcal{N}_{0},\overline\D)$ and $\varrho_{2}:\mathcal{J}^{2}\rightarrow H^{3}(\mathcal{N}_{0},\overline\D)$, related to the third cohomology of $(\mathcal{N}_{0},\overline\D)$. Also, the submodule $\mathcal{Z}^{2}_{2}$ is related to the 3-coboundaries of $(\mathcal{N}_{1},\overline\D)$.

\begin{corollary}\label{cor:H2}
  In the case when $\mathcal{R}$ is a field, the coboundary, cocycle, and cohomology spaces of degree $2$ admit the following splittings as vector spaces:
  \begin{align*}
  B^{2}(\mathcal{C},\D) &\cong (B^{2}(\mathcal{C},\D)\cap\mathcal{C}^{2,0})\oplus(\mathcal{B}^{2}_{1}\cap\mathcal{C}^{1,1})\oplus B^{2}(\mathcal{C}^{0,\bullet},\D_{0,1}), \\
  Z^{2}(\mathcal{C},\D)&\cong Z^{2}(\mathcal{N}_{0},\overline\D)\oplus\ker(\varrho_{2})\oplus\mathcal{Z}^{2}_{2}, \\
  H^{2}(\mathcal{C},\D)&\cong \frac{Z^{2}(\mathcal{N}_{0},\overline\D)}{B^{2}(\mathcal{C},\D)\cap\mathcal{C}^{2,0}} \oplus\frac{\ker(\varrho_{2})}{\mathcal{B}^{2}_{1}\cap\mathcal{C}^{1,1}}\oplus\frac{\mathcal{Z}^{2}_{2}}{B^{2}(\mathcal{C}^{0,\bullet},\D_{0,1})}.
  \end{align*}
\end{corollary}

In a more explicit fashion, the modules appearing in Theorem \ref{teo:H2} are
\[
B^{2}(\mathcal{C},\D)\cap\mathcal{C}^{2,0} &= \{\D_{1,0}\alpha+\D_{2,-1}Y \mid \alpha\in\mathcal{C}^{1,0}, Y\in\mathcal{C}^{0,1}, \D_{0,1}\alpha+\D_{1,0}Y=0,\D_{0,1}Y=0\},\\
\mathcal{B}^{2}_{1}\cap\mathcal{C}^{1,1} &= \{\D_{1,0}Y+\D_{0,1}\alpha \mid \alpha\in\mathcal{C}^{1,0}, Y\in\mathcal{C}^{0,1}, \D_{0,1}Y=0\},\\ B^{2}(\mathcal{C}^{0,\bullet},\D_{0,1}) &= \{\D_{0,1}Y \mid Y\in\mathcal{C}^{0,1}\},\\
\mathcal{J}^{2} &= \{Q\in\mathcal{C}^{1,1} \mid \D_{0,1}Q=0, \exists\beta_{Q}\in\mathcal{C}^{2,0}:\D_{0,1}\beta_{Q}+\D_{1,0}Q=0\},\\
Z^{2}(\mathcal{N}_{0},\overline\D) &= \{\beta\in\mathcal{C}^{2,0} \mid \D_{1,0}\beta=0, \D_{0,1}\beta=0\},\\
\ker(\varrho_{2}) &= \{Q\in\mathcal{J}^{2} \mid \D_{2,-1}Q + \D_{1,0}\beta_{Q}\in B^{3}(\mathcal{N},\overline\D)\},\\
\mathcal{Z}^{2}_{2} &=
\left\{V\in\ker^{0,2}\D_{0,1}~\left|~\exists Q\in\mathcal{C}^{1,1},\beta\in\mathcal{C}^{2,0}:\begin{aligned}
&\D_{0,1}Q + \D_{1,0}V = 0,\\
&\D_{0,1}\beta+\D_{1,0}Q+\D_{2,-1}V \in B^3(\mathcal{N},\overline\D),\\
&\D_{1,0}\beta+\D_{2,-1}Q\in B^3(\mathcal{N},\overline\D).
\end{aligned}\right.
\right\}.
\]

\pagebreak

\paragraph{Third cohomology.} Finally, the following result gives a more explicit presentation of the $\mathcal{R}$-modules involved in the description of coboundaries, cocycles, and cohomology of degree 3.

\begin{theorem}\label{teo:H3}
We have the following commutative diagrams with exact rows and columns,
\[
\xymatrix@=1.4em{
&0\ar[d]&0\ar[d]&0\ar[d]\\
0\ar[r]&B^{3}(\mathcal{C},\D)\cap\mathcal{C}^{3,0}\ar@{^{(}->}[d]\ar@{^{(}->}[r]&B^{3}(\mathcal{C},\D)\ar@{^{(}->}[d]\ar[r]^{\pi_{1}}&\mathcal{B}^{3}_{1}\ar@{^{(}->}[d]\ar[r]&0\\
0\ar[r]&Z^{3}(\mathcal{N}_{0},\overline\D)\ar[d]\ar@{^{(}->}[r]&Z^{3}(\mathcal{C},\D)\ar[d]\ar[r]^{\pi_{1}}&\ker(\rho_{3})\ar[d]\ar[r]&0,\\
0\ar[r]&\frac{Z^{3}(\mathcal{N}_{0},\overline\D)}{B^{3}(\mathcal{C},\D)\cap\mathcal{C}^{3,0}}\ar[d]\ar@{->}[r]&H^{3}(\mathcal{C},\D)\ar[d]\ar@{->}[r]&\frac{\ker(\rho_{3})}{\mathcal{B}^{3}_{1}}\ar[d]\ar[r]&0\\
&0&0&0}
\]
\[
\xymatrix@=1.4em{
&0\ar[d]&0\ar[d]&0\ar[d]\\
0\ar[r]&\mathcal{B}^{3}_{1}\cap\mathcal{C}^{2,1}\ar@{^{(}->}[d]\ar@{^{(}->}[r]&\mathcal{B}^{3}_{1}\ar@{^{(}->}[d]\ar[r]^{\pi_{2}}&\mathcal{B}^{3}_{2}\ar@{^{(}->}[d]\ar[r]&0\phantom{.}\\
0\ar[r]&\ker(\varrho_{3})\ar[d]\ar@{^{(}->}[r]&\ker(\rho_{3})\ar[d]\ar[r]^{\pi_{2}}&\mathcal{Z}^{3}_{2}\ar[d]\ar[r]&0,\\
0\ar[r]&\frac{\ker(\varrho_{3})}{\mathcal{B}^{3}_{1}\cap\mathcal{C}^{2,1}}\ar[d]\ar@{->}[r]&\frac{\ker(\rho_{3})}{\mathcal{B}^{3}_{1}}\ar[d]\ar@{->}[r]&\frac{\mathcal{Z}^{3}_{2}}{\mathcal{B}^{3}_{2}}\ar[d]\ar[r]&0\phantom{.}\\
&0&0&0}
\]
\[
\xymatrix@=1.4em{
&0\ar[d]&0\ar[d]&0\ar[d]\\
0\ar[r]&\mathcal{B}^{3}_{2}\cap\mathcal{C}^{1,2}\ar@{^{(}->}[d]\ar@{^{(}->}[r]&\mathcal{B}^{3}_{2}\ar@{^{(}->}[d]\ar[r]^(0.3){\pi_{3}}&B^{3}(\mathcal{C}^{0,\bullet},\D_{0,1})\ar@{^{(}->}[d]\ar[r]&0\phantom{.}\\
0\ar[r]&\mathcal{Z}^{3}_{2}\cap\mathcal{C}^{1,2}\ar[d]\ar@{^{(}->}[r]&\mathcal{Z}^{3}_{2}\ar[d]\ar[r]^{\pi_{3}}&\mathcal{Z}^{3}_{3}\ar[d]\ar[r]&0.\\
0\ar[r]&\frac{\mathcal{B}^{3}_{2}\cap\mathcal{C}^{1,2}}{\mathcal{Z}^{3}_{2}\cap\mathcal{C}^{1,2}}\ar[d]\ar@{->}[r]&\frac{\mathcal{Z}^{3}_{2}}{\mathcal{B}^{3}_{2}}\ar[d]\ar@{->}[r]&\frac{\mathcal{Z}^{3}_{3}}{B^{3}(\mathcal{C}^{0,\bullet},\D_{0,1})}\ar[d]\ar[r]&0\phantom{.}\\
&0&0&0}
\]
\end{theorem}

We note that this result on the third cohomology of $(\mathcal{C},\D)$ involves the maps $\rho_{3}:\mathcal{A}^{3}\rightarrow H^{4}(\mathcal{N}_{0},\overline\D)$ and $\varrho_{3}:\mathcal{J}^{3}\rightarrow H^{4}(\mathcal{N}_{0},\overline\D)$, related to the fourth cohomology of $(\mathcal{N}_{0},\overline\D)$. Also, the submodule $\mathcal{Z}^{3}_{3}$ is related to the 4-coboundaries of $(\mathcal{N}_{1},\overline\D)$.

\begin{corollary}\label{cor:H3}
  In the case when $\mathcal{R}$ is a field, the coboundary, cocycle, and cohomology spaces of degree $3$ admit the following splittings as vector spaces:
  \begin{align*}
  &B^{3}(\mathcal{C},\D) \cong (B^{3}(\mathcal{C},\D)\cap\mathcal{C}^{3,0}) \oplus (\mathcal{B}^{3}_{1}\cap\mathcal{C}^{2,1}) \oplus(\mathcal{B}^{3}_{2}\cap\mathcal{C}^{1,2})\oplus B^{3}(\mathcal{C}^{0,\bullet},\D_{0,1}), \\
  &Z^{3}(\mathcal{C},\D)\cong Z^{3}(\mathcal{N}_{0},\overline\D) \oplus \ker(\varrho_{3}) \oplus (\mathcal{Z}^{3}_{2}\cap\mathcal{C}^{1,2}) \oplus \mathcal{Z}^{3}_{3}, \\
  &H^{3}(\mathcal{C},\D)\cong \frac{Z^{3}(\mathcal{N}_{0},\overline\D)}{B^{3}(\mathcal{C},\D)\cap\mathcal{C}^{3,0}} \oplus \frac{\ker(\varrho_{3})}{\mathcal{B}^{3}_{1}\cap\mathcal{C}^{2,1}} \oplus \frac{\mathcal{Z}^{3}_{2}\cap\mathcal{C}^{1,2}}{\mathcal{B}^{3}_{2}\cap\mathcal{C}^{1,2}} \oplus \frac{\mathcal{Z}^{3}_{3}}{B^{3}(\mathcal{C}^{0,\bullet},\D_{0,1})}.
  \end{align*}
\end{corollary}

Each of the terms appearing in the splittings of Corollary \ref{cor:H3} are given as follows:
\[
& B^{3}(\mathcal{C},\D)\cap\mathcal{C}^{3,0} = \left\{\D_{1,0}\beta+\D_{2,-1}Q ~\left|~ 
\beta\in\mathcal{C}^{2,0},
Q\in\mathcal{C}^{1,1},
\exists V\in\mathcal{C}^{0,2}:
\begin{aligned}
  &\D_{0,1}\beta+\D_{1,0}Q+\D_{2,-1}V=0,\\
  &\D_{0,1}Q+\D_{1,0}V=0,\\
  &\D_{0,1}V=0.
\end{aligned}\right.\right\},\\
&\mathcal{B}^{3}_{1}\cap\mathcal{C}^{2,1}=\{\D_{0,1}\beta+\D_{1,0}Q+\D_{2,-1}V\mid\beta\in\mathcal{C}^{2,0},Q\in\mathcal{C}^{1,1},V\in\mathcal{C}^{0,2},~\D_{0,1}Q+\D_{1,0}V=0,\D_{0,1}V=0\},\\
&\mathcal{B}^{3}_{2}\cap\mathcal{C}^{1,2} = \{\D_{0,1}Q+\D_{1,0}V=0 \mid Q\in\mathcal{C}^{1,1},V\in\mathcal{C}^{0,2},~\D_{0,1}V=0\},\\
& B^{3}(\mathcal{C}^{0,\bullet},\D_{0,1}) = \{\D_{0,1}V \mid V\in\mathcal{C}^{0,2}\},\\
&\mathcal{J}^{3} = \{R\in\mathcal{C}^{2,1} \mid \D_{0,1}R=0, \exists\varphi_{R}\in\mathcal{C}^{3,0}:\D_{0,1}\varphi_{R}+\D_{1,0}R=0\},\\
&Z^{3}(\mathcal{N}_{0},\overline\D) = \{\varphi\in\mathcal{C}^{3,0} \mid \D_{1,0}\varphi=0, \D_{0,1}\varphi=0\},\\
&\ker(\varrho_{3}) = \{R\in\mathcal{J}^{3} \mid \D_{2,-1}R+\D_{1,0}\varphi_{R}\in B^{3}(\mathcal{N}_{0},\overline\D)\},\\
&\mathcal{Z}^{3}_{2}\cap\mathcal{C}^{1,2} = \left\{S\in\mathcal{C}^{1,2} ~\left|~ \exists R\in\mathcal{C}^{2,1},\varphi\in\mathcal{C}^{3,0}:\begin{aligned}
&\D_{0,1}S=0,  \qquad \D_{0,1}R+\D_{1,0}S=0,\\
&\D_{0,1}\varphi+\D_{1,0}R+\D_{2,-1}S\in B^{4}(\mathcal{N}_{0},\overline\D),\\
&\D_{1,0}\varphi+\D_{2,-1}R\in B^{4}(\mathcal{N}_{0},\overline\D).
\end{aligned}\right.\right\},\\
&\mathcal{Z}^{3}_{3} = \left\{W\in\ker^{0,3}(\D_{0,1}) ~\left|~ \exists S\in\mathcal{C}^{1,2},R\in\mathcal{C}^{2,1},\varphi\in\mathcal{C}^{3,0}:\begin{aligned}
&\D_{0,1}S+\D_{1,0}W=0,\\
&\D_{0,1}R+\D_{1,0}S+\D_{2,-1}W\in B^{4}(\mathcal{N}_{0},\overline\D),\\
&\D_{0,1}\varphi+\D_{1,0}R+\D_{2,-1}S\in B^{4}(\mathcal{N}_{0},\overline\D),\\
&\D_{1,0}\varphi+\D_{2,-1}R\in B^{4}(\mathcal{N}_{0},\overline\D).
\end{aligned}\right.\right\}.
\]

The proofs of Theorems \ref{teo:H2}, and \ref{teo:H3} are analogous to the proof of Theorem \ref{teo:H1}.

\section{Particular cases}\label{sec:BigradedParticular}

In this part we consider some particular cases regarding the bigraded cochain complex $(\mathcal{C},\D)$,
\[
\D = \D_{0,1} + \D_{1,0} + \D_{2,-1}.
\]

\paragraph{The case } $\D_{2,-1}=0$\textbf{.} In this part, we assume that $\D_{2,-1}=0$, which corresponds to the well-known case of a \emph{double complex}. Namely, $(\mathcal{C},\D)$ is a bigraded cochain complex, such that the bigraded decomposition of the coboundary operator is
\[
\D = \D_{0,1} + \D_{1,0}.
\]
The coboundary equations \eqref{eq:Cob1}-\eqref{eq:Cob5} read in this case
\[
\D_{0,1}^{2} = 0, && \D_{0,1}\D_{1,0}+\D_{1,0}\D_{0,1} = 0, && \D_{1,0}^{2}=0,
\]
which means that the bigraded components $\D_{0,1}$ and $\D_{1,0}$ are coboundary operators which commute with each other in the graded sense.

The case of the double complex is a standard topic in the literature, since it naturally arises both from algebraic and geometric contexts \cite[Chapter XI, Section 6]{Maclane}, \cite[Section 2.4]{McCleary}, and has several applications \cite[Chapter II]{BottTu}. However, the description of its cohomology is limited to explain that the natural filtration
\[
F^{p}\mathcal{C}^{\bullet} := \bigoplus_{\substack{i,j\in\mathbb{Z}\\i\geq p}}\mathcal{C}^{i,j}
\]
induces a spectral sequence which converges to the cohomology, and whose second page is explicitly described in terms of the double complex, namely, $E^{p,q}_{2} = H^{p}(H^{q}(\mathcal{C},\D_{0,1}),\D_{1,0})$. For several applications discussed in the literature, the computation of the second page of the spectral sequence is sufficient to describe the cohomology. In this sense, we have not found a general scheme for the computation of the cohomology of a double complex.

Theorems \ref{teo:H1}, \ref{teo:H2}, and \ref{teo:H3} provide an explicit description of the low-degree cohomology of a bigraded cochain complex which, of course, also holds for the double complex. We remark that in this case, the null subcomplex is simply $\mathcal{N}=\ker\D_{0,1}$, so the cocycles and coboundaries of $(\mathcal{N},\overline\D)$ are
\[
Z(\mathcal{N},\overline\D)=\ker\D_{0,1}\cap\ker\D_{1,0},  && \text{and} &&  B(\mathcal{N},\overline\D)=\D_{1,0}(\ker\D_{0,1}).
\]
Furthermore, in the description of the cohomology of degree 1 provided by Theorem \ref{teo:H1}, we have
\[
\ker(\rho_{1}) = \{Y\in\mathcal{C}^{0,1} \mid \D_{0,1}Y = 0, ~\exists \alpha_{Y}\in\mathcal{C}^{1,0}~:~ \D_{1,0}Y+\D_{0,1}\alpha_{Y}=0, \D_{1,0}\alpha_{Y}\in B^{1}(\mathcal{N}_{0},\overline\D)\}.
\]

On the other hand, the terms which simplify in the description of the cohomology of degree 2 of a double complex are
\[
B^{2}(\mathcal{C},\D)\cap\mathcal{C}^{2,0} &= \{\D_{1,0}\alpha \mid \alpha\in\mathcal{C}^{1,0}, ~\exists Y\in\mathcal{C}^{0,1} ~:~ \D_{0,1}\alpha+\D_{1,0}Y=0,\D_{0,1}Y=0\},\\
\ker(\varrho_{2}) &= \{Q\in\mathcal{C}^{1,1} \mid \D_{0,1}Q=0, \exists\beta_{Q}\in\mathcal{C}^{2,0}:\D_{0,1}\beta_{Q}+\D_{1,0}Q=0, \D_{1,0}\beta_{Q}\in B^{3}(\mathcal{N},\overline\D)\},\\
\mathcal{Z}^{2}_{2} &=
\left\{V\in\ker^{0,2}\D_{0,1}~\left|~\exists Q\in\mathcal{C}^{1,1},\beta\in\mathcal{C}^{2,0}:\begin{aligned}
&\D_{0,1}Q + \D_{1,0}V = 0,\\
&\D_{0,1}\beta+\D_{1,0}Q\in B^3(\mathcal{N},\overline\D),\\
&\D_{1,0}\beta\in B^3(\mathcal{N},\overline\D).
\end{aligned}\right.
\right\}.
\]

In a similar fashion, the terms appearing in the description of the cohomology of degree three that simplify in the case of the double complex are $\mathcal{B}^{3}_{1}\cap\mathcal{C}^{2,1}$, $\mathcal{B}^{3}_{2}\cap\mathcal{C}^{1,2}$, $\mathcal{Z}^{3}_{2}\cap\mathcal{C}^{1,2}$, and $\mathcal{Z}^{3}_{3}$.

\paragraph{The case } $\D_{0,1}=0$\textbf{.} We now consider a cochain complex $(\mathcal{C}^{\bullet},\D)$ endowed with a compatible bigrading such that the decomposition of the coboundary operator is of the form $\D = \D_{1,0} + \D_{2,-1}$. This can be regarded as a particular case of our general scheme in which the operator of type $(0,1)$ vanishes, $\D_{0,1}=0$. In this case, the coboundary property $\D^{2}=0$ is equivalent to
\[
\D_{1,0}^{2} = 0, && \D_{1,0}\D_{2,-1}+\D_{2,-1}\D_{1,0} = 0, && \D_{2,-1}^{2}=0,
\]
which means that the bigraded components $\D_{1,0}$ and $\D_{2,-1}$ are graded commutative coboundary operators. Moreover, the null subcomplex is $\mathcal{N}=\ker\D_{2,-1}$. In particular, $(\mathcal{N}_{0},\overline{\D})=(\mathcal{C}^{\bullet,0},\D_{1,0})$,
\[
Z^{p}(\mathcal{N}_{0},\overline\D)=\ker(\D_{1,0}:\mathcal{C}^{p,0}\rightarrow\mathcal{C}^{p+1,0}),  && \text{and} &&  B^{p}(\mathcal{N}_{0},\overline\D)=\D_{1,0}(\mathcal{C}^{p-1,0}), &&\forall p\geq 0.
\]

It is well known that this class of cochain complexes arise in the context of regular Poisson manifolds. In fact, the choice of a subbundle normal to the symplectic foliation of a regular Poisson manifold induces a bigrading of the Lichnerowicz-Poisson complex such that the coboundary operator is of this kind. Moreover, based on this fact, and motivated by the results in \cite{VK-88}, a recursive scheme for the computation of the cohomology of regular Poisson manifolds is provided in \cite[Section 2]{Va-90}. Such recursive scheme is similar to the one we have presented in Section \ref{sec:BigradedRecursive}, and leads to a description of the Poisson cohomology in terms of short exact sequences that coincide with the bottom rows of the diagrams of Proposition \ref{prop:Diagram}.

Finally, we remark that this class of cochain complexes also arise in the literature in the more general context of Poisson foliations \cite[Proposition 2.2]{Va-04}, \cite[Lemma 4.1]{Va-05}, which are Poisson structures such that the symplectic foliation admits an outer regularization.

The terms which appear in the description of the first cohomology which simplify in this case are
\begin{align*}
B^{1}(\mathcal{N}_{0},\overline\D) &= \D_{1,0}(\mathcal{C}^{0,0}),\\
B^{1}(\mathcal{C}^{0,\bullet},\D_{0,1}) &= \{0\},\\
\mathcal{A}^{1} &= \{Y\in\mathcal{C}^{0,1}\mid\D_{1,0}Y=0\},\\
Z^{1}(\mathcal{N}_{0},\overline\D) &= \{\alpha\in\mathcal{C}^{1,0} \mid \D_{1,0}\alpha=0\},\\
\ker(\rho_{1}) &= \{Y\in\mathcal{A}^{1} \mid \D_{2,-1}Y \in \D_{1,0}(\mathcal{C}^{0,0})\}.
\end{align*}

In the description of the cohomology of degree two, the terms simplify to
\[
B^{2}(\mathcal{C},\D)\cap\mathcal{C}^{2,0} &= \{\D_{1,0}\alpha+\D_{2,-1}Y \mid \alpha\in\mathcal{C}^{1,0}, Y\in\mathcal{C}^{0,1}, \D_{1,0}Y=0\},\\
\mathcal{B}^{2}_{1}\cap\mathcal{C}^{1,1} &= \{\D_{1,0}Y \mid Y\in\mathcal{C}^{0,1}\},\\ B^{2}(\mathcal{C}^{0,\bullet},\D_{0,1}) &= \{0\},\\
\mathcal{J}^{2} &= \{Q\in\mathcal{C}^{1,1} \mid \D_{1,0}Q=0\},\\
Z^{2}(\mathcal{N}_{0},\overline\D) &= \{\beta\in\mathcal{C}^{2,0} \mid \D_{1,0}\beta=0\},\\
\ker(\varrho_{2}) &= \{Q\in\mathcal{J}^{2} \mid \D_{2,-1}Q + \D_{1,0}\beta_{Q}\in B^{3}(\mathcal{N},\overline\D)\},\\
\mathcal{Z}^{2}_{2} &=
\left\{V\in\mathcal{C}^{0,2}~\left|~\exists Q\in\mathcal{C}^{1,1},\beta\in\mathcal{C}^{2,0}:\begin{aligned}
&\D_{1,0}V = 0,\\
&\D_{1,0}Q+\D_{2,-1}V \in B^3(\mathcal{N},\overline\D),\\
&\D_{1,0}\beta+\D_{2,-1}Q\in B^3(\mathcal{N},\overline\D).
\end{aligned}\right.
\right\}.
\]

In a similar fashion, most of the terms appearing in the description of the cohomology of degree three simplify.

\section{Geometric Applications}

\paragraph{An example.} Let $(M,\omega)$ be a presymplectic manifold and $(N,\Psi)$ the Poisson manifold given by $N=\mathbb{R}^{2}_{y}$, $\Psi=\|y\|^{2}\tfrac{\partial}{\partial y_{1}}\wedge\tfrac{\partial}{\partial y_{2}}$. Consider the product Dirac structure $D$ on $M\times N$. Then, $M\times\{0\}$ is a presymplectic leaf of $D$ and we can think of $M\times N\overset{\operatorname{pr}_{M}}{\rightarrow}M$ as a coupling neighborhood. More precisely, the vertical distribution is $\mathbb{V}:=\ker(\operatorname{pr}_{M})_{*}$, and the associated geometric data $(\gamma,\sigma,P)$ consists of the flat connection $\gamma:=(\operatorname{pr}_{N})_{*}$, given by the differential of the projection $\operatorname{pr}_{N}:M\times N\rightarrow N$; the pullback $\sigma:=\operatorname{pr}_{N}^{*}\omega$ of the presymplectic structure on $M$; and the unique vertical Poisson bivector field $P$ on $M\times N$ which is $\operatorname{pr}_{N}$-related to $\Psi$, $P=\|y\|^{2}\tfrac{\partial}{\partial y_{1}}\wedge\tfrac{\partial}{\partial y_{2}}$ on $M\times N$.

First we note that $(N,\Psi)$ has two symplectic leaves: the origin $(0,0)$, which is zero-dimensional, and the complement $N_{\mathrm{reg}}:=\mathbb{R}^{2}-\{(0,0)\}$. Then, each Casimir funcion of $(N,\Psi)$ is constant, $\operatorname{Casim}(N,\Psi)\cong\mathbb{R}$. This implies that $\operatorname{Casim}(M\times N,P)=\operatorname{pr}_{M}^{*}C^{\infty}(M)$ and
\[
H^{0}(M\times N,D) \cong H^{0}_{\mathrm{dR}}(M).
\]
Observe that in this case, the de Rham - Casimir complex $(\mathcal{N}^{\bullet},\overline\D)$ is isomorphic to the de Rham complex $(\Gamma(\wedge^{\bullet}T^{*}M),\operatorname{d})$ of $M$. In particular, $H^{1}(\mathcal{N}^{\bullet},\overline\D)\cong H^{1}_{\mathrm{dR}}(M)$. We now proceed to describe $\ker(\rho_{1}:\mathcal{I}^{1}\rightarrow H^{2}(\mathcal{N}^{\bullet},\overline\D))$. Let $Y=Y_{1}\tfrac{\partial}{\partial y_{1}}+Y_{2}\tfrac{\partial}{\partial y_{2}}\in\Gamma(\mathbb{V})$ be a vertical vector field. Then, $\operatorname{L}_{Y}P=0$ if and only if
\begin{equation}\label{eq:Ej1-Poiss}
  y_{1}Y_{1}+y_{2}Y_{2} = \tfrac{1}{2}\|y\|^{2}\operatorname{div}^{y}(Y).
\end{equation}
Here, $\operatorname{div}^{y}(Y):=\tfrac{\partial Y_{1}}{\partial y_{1}}+\tfrac{\partial Y_{2}}{\partial y_{2}}$ denotes the divergence of $Y$ with respect to the fiber-wise volume form $\operatorname{d}y_{1}\wedge\operatorname{d}y_{2}$. In particular, the fiber-wise Euler and modular vector fields
\begin{equation}\label{eq:Ej1-EulerMod}
Z_{1}:=y_{1}\tfrac{\partial}{\partial y_{1}}+y_{2}\tfrac{\partial}{\partial y_{2}} \quad\text{and}\quad Z_{2}:=-y_{2}\tfrac{\partial}{\partial y_{1}}+y_{1}\tfrac{\partial}{\partial y_{2}}
\end{equation}
are Poisson vector fields of $P$, respectively.

\begin{proposition}\label{prop:Ej1}
  Consider the Poisson manifold $(M\times N,P)$ given as in above. Let $Z_{1},Z_{2}$ be defined as in \eqref{eq:Ej1-EulerMod}, and $Y\in\Gamma(\mathbb{V})$. Then, $Y\in\operatorname{Poiss}(M\times N,P)$ if and only if $Y=a_{1}Z_{1}+a_{2}Z_{2}$ for unique $a_{1},a_{2}\in C^{\infty}(M)$ satisfying $L_{Z_{1}}a_{1}+L_{Z_{2}}a_{2}=0$. Additionally, $Y$ is Hamiltonian if and only if $a_{1}$ and $a_{2}$ vanish along the zero section $M\times\{0\}$. In this case, a Hamiltonian is given by $h(x,y):=\int_{0}^{\infty}a_{2}(x,e^{-t}y)\operatorname{d}t$. Hence, the first vertical Poisson cohomology of the Poisson bundle $(M\times N\overset{\operatorname{pr}_{M}}{\rightarrow}M,P)$ is
  \[
  H^{1}(M\times N,\mathbb{V},P)=C^{\infty}(M).[Z_{1}]\oplus C^{\infty}(M).[Z_{2}].
  \]
\end{proposition}

\begin{proof}
Observe from \eqref{eq:Ej1-Poiss} that, for smooth functions $a_{1},a_{2}\in C^{\infty}(M\times N)$,
\begin{equation}\label{eq:Ej1-a1a2}
  Y=a_{1}Z_{1}+a_{2}Z_{2}
\end{equation}
is an infinitesimal Poisson automorphism of $P$ if and only if $\operatorname{L}_{Z_{1}}a_{1}+\operatorname{L}_{Z_{2}}a_{2}=0$. We claim that every vertical infinitesimal Poisson automorphism of $P$ is of this form. Indeed, fix $Y\in\operatorname{Poiss}(M\times N,P)\cap\Gamma(\mathbb{V})$. Since $Z_{1},Z_{2}$ are linearly independent on $M\times N_{\mathrm{reg}}$, there exist $b_{1},b_{2}\in C^{\infty}(M\times N_{\mathrm{reg}})$ such that $Y|_{M\times N_{\mathrm{reg}}}=b_{1}Z_{1}+b_{2}Z_{2}$. Explicitly, $b_{1}:=\tfrac{1}{\|y\|^{2}}(y_{1}Y_{1}+y_{2}Y_{2})$ and $b_{2}:=\tfrac{1}{\|y\|^{2}}(-y_{2}Y_{1}+y_{1}Y_{2})$. We just need to show that $b_{1}$ and $b_{2}$ can be extended to some smooth functions $a_{1}$ and $a_{2}$ on $M\times N$. By \eqref{eq:Ej1-Poiss}, $b_{1}$ can be extended to the smooth function $a_{1}:=\tfrac{1}{2}\operatorname{div}^{y}(Y)$ on $M\times N$. Moreover, $Y_{2} - a_{1}y_{2}$ is a smooth function on $M\times N$ such that $Y_{2}-a_{1}y_{2}|_{M\times N_{\mathrm{reg}}} = b_{2}y_{1}$. This implies that $Y_{2}-a_{1}y_{2}$ vanishes along the level set $y_{1}=0$ of $M\times N$. Hence, there exists $a_{2}\in C^{\infty}(M\times N)$ such that $Y_{2}-a_{1}y_{2} = a_{2}y_{1}$. Hence, $a_{2}$ is a smooth function on $M\times N$ whose restriction to $M\times N_{\mathrm{reg}}$ is $b_{2}$. Finally, since such extensions are clearly unique, we conclude that every Poisson vector field $Y$ admits a unique representation of the form \eqref{eq:Ej1-a1a2}, with $a_{1},a_{2}\in C^{\infty}(M\times N)$ satisfying $\operatorname{L}_{Z_{1}}a_{1}+\operatorname{L}_{Z_{2}}a_{2}=0$.

Now, pick a Hamiltonian vector field $Y=P^{\sharp}\operatorname{d}h$, $h\in C^{\infty}(M\times N)$. By straightforward computations, the smooth functions $a_{1}$ and $a_{2}$ in \eqref{eq:Ej1-a1a2} are given in this case by
\begin{equation}\label{eq:Ej1-Ham}
  a_{1} = -\operatorname{L}_{Z_{2}}h \quad\text{and}\quad a_{2} = \operatorname{L}_{Z_{1}}h.
\end{equation}
In particular, $a_{1}$ and $a_{2}$ vanish on the symplectic leaf $M\times\{0\}$. We now see that the converse is also true: If $Y=a_{1}Z_{1}+a_{2}Z_{2}\in\operatorname{Poiss}(M\times N,P)\cap\Gamma(\mathbb{V})$ is such that $a_{1}(x,0)=a_{2}(x,0)=0$, then $Y\in\operatorname{Ham}(M\times N,P)$. To see this, pick $Y=a_{1}Z_{1}+a_{2}Z_{2}$ such that $\operatorname{L}_{Z_{1}}a_{1}+\operatorname{L}_{Z_{2}}a_{2}=0$ with $a_{1},a_{2}$ vanishing on $M\times\{0\}$. Define $h:M\times N\rightarrow\mathbb{R}$ by
\begin{equation}\label{eq:Ej1-hInt}
  h(x,y):=\int_{0}^{\infty}a_{2}(x,e^{-t}y)\operatorname{d}t.
\end{equation}
Then, $h\in C^{\infty}(M\times N)$ and clearly satisfies $\operatorname{L}_{Z_{1}}h = a_{2}$. Moreover,
\[
\operatorname{L}_{Z_{1}}a_{1} = -\operatorname{L}_{Z_{2}}a_{2} = -\operatorname{L}_{Z_{2}}\operatorname{L}_{Z_{1}}h = -\operatorname{L}_{Z_{1}}\operatorname{L}_{Z_{2}}h,
\]
so $\operatorname{L}_{Z_{1}}(a_{1}+\operatorname{L}_{Z_{2}}h)=0$, which implies that $a_{1}+\operatorname{L}_{Z_{2}}h$ is constant along the $\operatorname{pr}_{M}$-fibers. By hypothesis, both $a_{1}$ and $\operatorname{L}_{Z_{2}}h$ vanish on $M\times\{0\}$. Hence, $a_{1}+\operatorname{L}_{Z_{2}}h=0$. Therefore, $h$ is a solution of \eqref{eq:Ej1-Ham} and so is a Hamiltonian for $Y$.
\end{proof}



We now describe the subspace $\mathcal{I}^{1}$. For $Y\in\operatorname{Poiss}(M\times N,P)\cap\Gamma(\mathbb{V})$, we have that $Y\in\mathcal{I}^{1}$ if and only if there exists a horizontal 1-form $\alpha\in\Gamma(\mathbb{V}^{0})$ such that $\D^{\gamma}_{1,0}Y + \D^{P}_{0,1}\alpha =0$. Since $\gamma$ is the trivial (flat) connection, the horizontal $\operatorname{pr}_{M}$-projectable vector fields $u$ are locally written in the form $u=u_{i}\tfrac{\partial}{\partial x_{i}}$. Therefore, the relation between $Y$ and $\alpha$ reads $[u,Y]=-P^{\sharp}\operatorname{d}[\alpha(u)]$. Since $Y$ is a vertical infinitesimal automorphism of $P$, there exists $a_{1},a_{2}\in C^{\infty}(M\times N)$ such that \eqref{eq:Ej1-a1a2} holds and $Y=a_{1}Z_{1}+a_{2}Z_{2}$. Then, $[u,Y] = (\operatorname{L}_{u}a_{1})Z_{1}+(\operatorname{L}_{u}a_{2})Z_{2}$. Because of Proposition \ref{prop:Ej1}, $[u,Y]$ is Hamiltonian if and only if $a_{1}$ and $a_{2}$ are constant along the zero section $M\times\{0\}$. Hence,
\[
\tfrac{\mathcal{I}^{1}}{\operatorname{Ham}(M\times N,P)} = \mathbb{R}.[Z_{1}]\oplus\mathbb{R}.[Z_{2}].
\]
Observe also that the Hamiltonian function $\alpha(u)$ of $-[u,Y]$ can be given by the formula
\[
\alpha(u(x,y)) = -\int_{0}^{\infty}(\operatorname{L}_{u}a_{2})(x,e^{-t}y)\operatorname{d}t.
\]
Furthermore, we have $\alpha=-\partial^{\gamma}_{1,0}h$, where $h$ is given as in \eqref{eq:Ej1-hInt}. This follows from the identity
\[
(\operatorname{L}_{u}h)(x,y) = \int_{0}^{\infty}(\operatorname{L}_{u}a_{2})(x,e^{-t}y)\operatorname{d}t \qquad\forall u.
\]
Therefore, the flatness of $\gamma$ implies $\D^{\gamma}_{1,0}\alpha=-(\D^{\gamma}_{1,0})^{2}h=-\operatorname{L}_{R^{\gamma}}h=0$. Finally, we have $\D^{\sigma}_{2,-1}Y = \operatorname{L}_{Y}\operatorname{pr}_{M}^{*}\omega = 0$. In consequence, $\rho_{1}(Y)=0$ for all $Y\in\mathcal{I}^{1}$, so $\ker\rho_{1} = \mathcal{I}^{1}$.

\begin{theorem}\label{teo:Ej1-H1}
  The first cohomology of the Dirac manifold $M\times N$ given by the product of a presymplectic manifold $(M,\omega)$ with the Poisson manifold $(N=\mathbb{R}^{2}_{y},\Psi=\|y\|^{2}\tfrac{\partial}{\partial y_{1}}\wedge\tfrac{\partial}{\partial y_{2}})$ is
  \[
  H^{1}(M\times N,D) \cong H^{1}_{\mathrm{dR}}(M)\oplus H^{1}(N,\Psi).
  \]
\end{theorem}

\begin{proof}
  Because of our above discussion, $H^{1}(\mathcal{N}^{\bullet},\overline\D)\cong H^{1}_{\mathrm{dR}}(M)$, and $\tfrac{\ker\rho_{1}}{\operatorname{Ham}(M\times N,P)} = \mathbb{R}.[Z_{1}]\oplus\mathbb{R}.[Z_{2}]$. Therefore,
  \[
  H^{1}(M\times N,D) \cong H^{1}_{\mathrm{dR}}(M)\oplus(\mathbb{R}.[Z_{1}]\oplus\mathbb{R}.[Z_{2}]).
  \]
  Finally, the fact that $H^{1}(N,\Psi)\cong\mathbb{R}.[Z_{1}]\oplus\mathbb{R}.[Z_{2}]$ follows from Proposition \ref{prop:Ej1} with $M$ consisting of a single point (see also \cite{Mon-02}).
\end{proof}

\bibliographystyle{acm}
\bibliography{C:/Users/hamst/Documents/Bibliography}
\end{document}